\documentclass[reqno,11pt,a4paper]{amsart}
\usepackage[top=2.0cm,bottom=2.0cm,left=2.9cm,right=2.9cm]{geometry}

\usepackage{amsthm,amsmath,amssymb}
\usepackage{mathrsfs,amsfonts,dsfont,functan,extarrows,mathtools}
\usepackage[numbers,sort&compress]{natbib}
\usepackage[colorlinks]{hyperref}

\newtheorem{theorem}{Theorem}[section]

\newtheorem{lemma}[theorem]{Lemma}

\numberwithin{equation}{section}
\allowdisplaybreaks

\def\l{\left\langle}
\def\r{\right\rangle}

\def\eps{\varepsilon}

\def\p{\partial}
\def\div{{\rm div}}
\def\d{\mathop{}\!\mathrm{d}}
\def\no{\nonumber}
\def\T{\mathbb{T}}
\def\R{\mathbb{R}}

\def\u{u^\infty}
\def\v{v^\infty}
\def\pp{p^\infty}
\def\q{q^\infty}

\def\ou{\overline{u}}

\def\ov{\overline{v}}
\def\op{\overline{p}}
\def\oq{\overline{q}}

\newcommand{\abs}[1]{\left|#1\right|}
\newcommand{\skp}[2]{\left\langle #1,\, #2 \right\rangle}

\begin{document}

\title[Reversible Gray-Scott by EnVarA]
      {On a Reversible Gray-Scott Type System \\from Energetic Variational Approach \\and Its Irreversible Limit}
      
\author[J. Liang]{Jiangyan Liang}
\address[Jiangyan Liang]
  {\newline School of Mathematics and Statistics, Wuhan University, Wuhan, 430072, People’s Republic of China}
\email{ljymath@whu.edu.cn}

\author[N. Jiang]{Ning Jiang}
\address[Ning Jiang]
  {\newline School of Mathematics and Statistics, Wuhan University, Wuhan, 430072, People’s Republic of China}
\email{njiang@whu.edu.cn}

\author[C. Liu]{Chun Liu}
\address[Chun Liu]
  {\newline Department of Applied Mathematics, Illinois Institute of Technology, Chicago, IL 60616, USA}
\email{cliu124@iit.edu}

\author[Y. Wang]{Yiwei Wang}
\address[Yiwei Wang]
  {\newline Department of Applied Mathematics, Illinois Institute of Technology, Chicago, IL 60616, USA}
\email{ywang487@iit.edu}

\author[T.-F. Zhang]{Teng-Fei Zhang$^\dag$}
\address[Teng-Fei Zhang]{\newline School of Mathematics and Physics, China University of Geosciences, Wuhan, 430074, P. R. China}
\email{zhangtf@cug.edu.cn}

\thanks{$^\dag$Corresponding author.}

\begin{abstract}
  Most of the previous studies on the well-known Gray-Scott model view it as an irreversible chemical reaction system. In this paper, we derive a four-species reaction-diffusion system using the energetic variational approach based on the law of mass action. This is a reversible Gray-Scott type model, which has a natural entropy structure. We establish the local well-posedness of this system, and justify the limit to the corresponding irreversible Gray-Scott type system as some backward coefficients tend to zero. Furthermore, under some smallness assumption on the initial data, we obtain the global-in-time existence of classical solutions of the reversible system.

\vspace*{5pt}


\end{abstract}

\maketitle

\section{Introduction}\label{sec:intro}

The research of reaction-diffusion systems has been an active field, especially since Turing developed morphogenesis and instability theory for reaction–diffusion systems in \cite{Tur52,Ni-11b}. The dynamic instability of the reaction and transport processes can lead to pattern formation and periodic oscillations. The pattern dynamics has been extensively studied from the viewpoints of mathematics and experiments \cite{Pea93,D91,NU-99physD,NU-01physD}. These geometrical structures consist of stripes and/or dots, often known as dissipative structures. Moreover, Prigogine et al. proposed the concept of the dissipative structures in the framework of the non‐equilibrium thermodynamics \cite{GP77,KP14}.

In the present paper, we study the following reversible reaction-diffusion system involving four chemical species:
  \begin{align}\label{sys:Re-Gray-Scott} \tag{Re-GS}
  \begin{cases}
    u_{t} = d_u \Delta u - k_1^+ u v^2 + k_1^- v^3 - k_0^+ u + k_0^- q, \\
    v_{t} = d_v \Delta v + k_1^+ u v^2 - k_1^- v^3 - k_2^+ v + k_2^- p,\\
    p_{t} = d_p \Delta p + k_2^+ v - k_2^- p, \\
    q_{t} = d_q \Delta q + k_0^+ u - k_0^- q\,.
  \end{cases}
  \end{align}
The unknowns $(u\,,v\,,p\,,q)$ are molecular concentrations of reactants (or products) $U, V, P$ and $Q$, depending on the time $t \ge 0$ and the spatial position $x \in \Omega \subset \mathbb{R}^3$. The positive constants $d_\alpha$'s with $\alpha = u,v,p,q$ are diffusion coefficients of the corresponding species. Mathematically, the system \eqref{sys:Re-Gray-Scott} is a nonlinear parabolic system,  which describes the evolution of the molecular concentrations with the following chemical reaction schemes:
  \begin{align}\label{eq:cubic-cata-re}
    U + 2V \xrightleftharpoons[k_1^-]{k_1^+} 3V, \quad
    V \xrightleftharpoons[k_2^-]{k_2^+} P, \quad  U \xrightleftharpoons[k_0^-]{k_0^+} Q,
  \end{align}
where $k^\pm_i$ with $i=0, 1,2$ are the forward and backward rate coefficients of the $i\mbox{-}$th reaction. It is assumed that the chemical rates obey the so-called \emph{law of mass action} (LMA) in the chemical kinetics theory, which indicates that, the rate of a reaction process is proportional to the concentrations of the reactants. Therefore, denoting by $r_i$ ($i= 0, 1,2$) the total rates for the above two reactions, it follows,
  \begin{align}
    r_1 = k_1^+ u v^2 - k_1^- v^3, \quad r_2 = k_2^+ v - k_2^- p, \quad r_0 = k_0^+ u - k_0^- q\,.
  \end{align}

The system \eqref{sys:Re-Gray-Scott} can be viewed as a generalization of the classical, irreversible Gray-Scott model \cite{GS83, GS84}, which arises originally from the study of cubic auto-catalytic reactions in a continuously flowing, well-stirred gel reactor. Gray-Scott model is governed by the following irreversible system:
  \begin{align}\label{eq:ori-GS}
  \begin{cases}
    u_{t} = d_u \Delta u - k_1^+ u v^2 + \alpha(1 - u),  \\
    v_{t} = d_v \Delta v + k_1^+ u v^2 - k_2^+ v.
  \end{cases}
  \end{align}
This system is centered on a cubic auto-catalysis reaction, with the catalyst species $V$ decays slowly to an inert product $P$.  The Gray-Scott system \eqref{eq:ori-GS} corresponds to the chemical reaction schemes:
  \begin{align}\label{GS-model}
    U +2V \xlongrightarrow{k_1^+} 3V, \quad V \xlongrightarrow{k_2^+} P, \quad U \xlongrightarrow{\alpha} Q\,.
  \end{align}
The coefficient $\alpha$ is also the rate of the process that feeds $U$. 
Comparing to the chemical reactions \eqref{eq:cubic-cata-re}, the reactions in classical Gray-Scott model \eqref{GS-model} is irreversible.

The four species reversible Gray-Scott type reaction-diffusion system \eqref{sys:Re-Gray-Scott} bears some apparent similarities to the original, irreversible Gray-Scott system. We expect fruitful analytical properties of the solutions to \eqref{sys:Re-Gray-Scott}. The first step of the analytical study is to establish the well-posedness of the system  \eqref{sys:Re-Gray-Scott} in different domains (bounded domain with proper boundary conditions, periodic domain, or whole space) and in different framework of solutions (weak or classical solutions). Furthermore, it is interesting to justify the relation between the reversible Gray-Scott system \eqref{sys:Re-Gray-Scott} and the irreversible Gray-Scott system \label{eq:ori-GS}.

A motivation of studying the reversible Gray-Scott system \eqref{sys:Re-Gray-Scott} instead of the original Gray-Scott system \eqref{eq:ori-GS} is that the reversible system possess an energy (entropy) structure. Indeed, the four-species reversible Gray-Scott-like system \eqref{sys:Re-Gray-Scott} can be derived from the following entropy-entropy production law:
  \begin{align}\label{energy-dissipation law}
    & \frac{\d}{\d t} \int_{\Omega}
            \left[
              u (\ln \tfrac{u}{\bar u} -1 ) + v (\ln \tfrac{v}{\bar v} -1 )
            + p (\ln \tfrac{p}{\bar p} -1 ) + q (\ln \tfrac{q}{\bar q} -1 )
            \right]  \d x \no\\
    & + \int_{\Omega} \left(\tfrac{d_u}{u} \abs{\nabla u}^2 + \tfrac{d_v}{v} \abs{\nabla v}^2 + \tfrac{d_p}{p} \abs{\nabla p}^2 + \tfrac{d_q}{q} \abs{\nabla q}^2 \right) \d x  \\
    & + \int_{\Omega} \left[(k_0^+ u - k_0^- q) \ln \tfrac{k_0^+ u}{k_0^- q}
     +  (k_1^+ u v^2 - k_1^- v^3) \ln \tfrac{k_1^+ u}{k_1^- v} +  (k_2^+ v - k_2^- p) \ln \tfrac{k_2^+ v}{k_2^- p}\right] \d x \no\\
    & =0,\no
  \end{align}
where $(\overline{u}, \overline{v}, \overline{p}, \overline{q})$ is the positive constant solution (equilibrium) of the system \eqref{sys:Re-Gray-Scott} satisfying the \emph{detailed balance condition}, see \eqref{eq:equilibrium} below. The relation \eqref{energy-dissipation law} will be also referred as the basic energy-dissipation law in following contexts. The derivation process will be given in section 2, by using the energetic variational approach (EnVarA). We emphasize here that besides the reversible Gray-Scott system \eqref{sys:Re-Gray-Scott} is derived from the point view of energetic variation, so should be its boundary condition in the bounded domain (see \cite{LW-19arma}). Analytically, the entropy-entropy production law \eqref{energy-dissipation law} leads to a natural a priori estimate of the system \eqref{sys:Re-Gray-Scott}. However, the estimate \eqref{energy-dissipation law} is not enough to obtain the compactness of the weak solutions, which is unlike the incompressible Navier-Stokes equations. For simplicity, in this paper, we consider the classical solutions of \eqref{sys:Re-Gray-Scott} in torus or whole space. We leave the harder global weak solutions and the boundary condition issues to the future study.

Furthermore, we also will establish a reversible-irreversible limit from the four species reversible Gray-Scott-like system to an irreversible Gray-Scott-like system as some of backward rate coefficients $k_1^- $ and $k_2^-$ go to zero simultaneously. Again, we work in the classical solutions in the torus or whole space.

\subsection{Gray-Scott model: review of mathematical studies} 
\label{sub:reviews}

The Gray-Scott model is of great importance since it describes several experimentally observable autocatalytic reactions such as chloride-iodide-malonic acid reaction, arsenite-iodate reaction, and some enzyme reactions in biochemistry and biology. In particular, its complex pattern formation behavior associated to various range of parameters attracts many attentions of researchers from different disciplines. The self-replicating pulses and spots was firstly studied numerically by Pearson in \cite{Pea93}, see also \cite{LMPS94,RPP-94prl}. Doelman-Kaper-Zegeling \cite{DKZ-97nonli} constructed rigorously single and multiple pulse solutions, and later the stability was studied in \cite{DGK-98physD}. Hale-Peletier-Troy \cite{HPT-99aml,HPT-00siap} studied the existence and stability issues in the equal diffusivities case of $d_u=d_v$. Concerning the weak interaction regime in which both of $d_u, d_v \ll 1$ are in the same order, a skeleton structure of self-replicating dynamic and spatio-temporal chaos are analyzed in Nishiura-Ueyama \cite{NU-99physD,NU-01physD}. On the other hand, there are many researches on other types of patterns such as spike, stripe, ring and so on. For instance, existence, stability and pulse-splitting behavior are studied in Wei \cite{Wei-99nonl,Wei-01physD} and his collaborated works \cite{WW-03physD,KW-05ejam}, based on the so-called semi-strong interaction regime where small diffusivity ratio $d_v \ll d_u = \mathcal{O}(1)$ is assumed. Besides, some general models are introduced and studied in \cite{LW-03physD,HLW-05ima} in dimensions one and two, concerning a general auto-catalytic scheme $U + mV \to (m+1) V,\ V \to P$ with reaction rates $uv^m$ and $v^n$ (assuming $m>n>1$). Later the travelling wave solutions are studied for the general order model in the case of that without feeding, see \cite{CQ-09jde,CQZ-16jde,ZCQZ-18jdde}. More researches on other related models and topics can be found, for instance, in \cite{Z21,GW81,KO94,VSK94,WW04,GMW20,LA17,G12} and references therein.



Chemical reactions in reality are reversible processes \cite{GQ-13pre}. However, there are only a few researches in this direction, to the best of our knowledge. As a modification of the original irreversible Gray-Scott model, a reversible Gray-Scott model involving three reactants $(U,\,V,\,P)$ was introduced in \cite{MSY04}, based on the reversible reaction scheme \eqref{eq:cubic-cata-re} in which the product $P$ is not an inert substance any more. The existence and robustness of global attractor is studied later in \cite{You-10dcds,You-12jdde}, and the global attractor of a lattice reversible model is studied in \cite{JZY-12amc}.

The present paper focuses on the reversible chemical reactions, which are assumed to obey the \emph{law of mass action}. Usually speaking, a chemical reaction should not be viewed as a Newtonian mechanics \cite{GQ-16pre}. Wang-Liu-Liu-Eisenberg \cite{WLLE-20pre} showed recently a possible variational treatment on the reaction-diffusion process obeying LMA and the detailed balance condition. Their formulation provides a basis of coupling chemical reactions with other mechanical effects. The core of their treatment is based on a generalized notion of \emph{energetic variation approach} (EnVarA). The EnVarA is developed from seminal work of Rayleigh \cite{Ray1871} and Onsager \cite{Ons1931-1,Ons1931-2}, and has proven to be a powerful tool to deal with the couplings and competitions between different mechanisms in different scales. This approach has been successfully applied to model many systems, especially those in complex fluids, such as liquid crystals, polymeric fluids, phase field and ion channels, see the survey \cite{GKL-18notes} for more details. The EnVarA can also be used to study problems with boundary, especially for dynamical boundary conditions problems for the Cahn-Hilliard equation \cite{LW-19arma,KLLM-20m2an} where chemical reactions occurring at the boundary are taken into account. Recently, a micro-macro model for living polymeric fluids involving the reversible chemical reaction of breakage and reforming process is derived by EnVarA in Liu-Wang-Zhang \cite{LWZ-20axv,LWZ-21nnfm}, where the global existence near equilibrium is established.

The main reason that we consider the reversible reactions is due to the \emph{entropy-entropy production} structure exhibited in this case. Based on this entropy structure \eqref{energy-dissipation law} combined with the corresponding kinematic relations, we can derive by a general EnVarA the reversible Gray-Scott-like system \eqref{sys:Re-Gray-Scott}. The detailed derivation can be found in \S \ref{sec:derivation_ReGS} below.

Notice that more phenomena will arise when we consider different timescales for different reaction schemes in the \eqref{sys:Re-Gray-Scott} system.
The limit issues of some diffusion-reaction system with small parameter are proved by Evans \cite{evans80} and Gajewski-Sparing \cite{Gaje84}. Chen-Gao \cite{Chen00} considered the well-posedness of a free boundary problem arising from the limit of a FitzHugh-Nagumo system (a slow-diffusion fast-reaction system). Bisi-Conforto-Desvillettes \cite{BCD-07bull} justified rigorously the quasi-steady-state approximation used in chemistry. Mielke-Peletier-Stephan \cite{MPS-21nonli} considered the nonlinear systems satisfying LMA with slow and fast reactions.

In this paper, we will study some certain scaling limit on parameters, speaking specifically, the reversible-irreversible limit as the backward rate coefficients $k_1^- $ and $k_2^-$ go to zero simultaneously. By employing in the \eqref{sys:Re-Gray-Scott} system the parameters $k_1^- = k_2^- = \eps$, we get the following approximate system:
  \begin{align}\label{sys:Re-Gray-Scott-eps} \tag*{(Re-GS)$_\eps$}
  \begin{cases}
    u^\eps_t = d_u \Delta u^\eps - k_1^+ u^\eps (v^\eps)^2 + \eps (v^\eps)^3 - k_0^+ u^\eps + k_0^- q^\eps, \\
    v^\eps_t = d_v \Delta v^\eps + k_1^+ u^\eps (v^\eps)^2 - \eps (v^\eps)^3 - k_2^+ v^\eps + \eps p^\eps,\\
    p^\eps_t = d_p \Delta p^\eps + k_2^+ v^\eps - \eps p^\eps, \\
    q^\eps_t = d_q \Delta q^\eps + k_0^+ u^\eps - k_0^- q^\eps.
  \end{cases}
  \end{align}
As the small parameter $\eps$ goes to zero, we get the limit system, at least formally, that
  \begin{align}\label{sys:Ir-Gray-Scott}\tag{Ir-GS}
  \begin{cases}
    \u_t = d_u \Delta \u - k_1^+ \u (\v)^2 - k_0^+ \u + k_0^- \q, \\
    \v_t = d_v \Delta \v + k_1^+ \u (\v)^2 - k_2^+ \v, \\
    \pp_t = d_p \Delta \pp + k_2^+ \v, \\
    \q_t = d_q \Delta \q + k_0^+ \u - k_0^- \q,
  \end{cases}
  \end{align}
which will be referred as the irreversible Gray-Scott system in the following context.

It is worth mentioning to obtain the original Gray-Scott model \eqref{eq:ori-GS}, we need to take another limiting from the irreversible Gray-Scott system \eqref{sys:Ir-Gray-Scott}, which is more challenging.
Formally, one can view the reaction
$$\textstyle U \xrightleftharpoons[k_0^-]{k_0^+} Q$$
as a birth-death reaction, which describes the exchange of the system with the environment. We assume the concentration relation in \eqref{sys:Ir-Gray-Scott}: $\q \gg \u$, and define a new unknown $s^\infty = \lambda \q$ with respect to a sufficiently small parameter $\lambda$. For the sake of exposition, we omit temporarily the molecular diffusion of $\pp$ and $\q$, i.e., $d_p = d_q = 0$. Therefore, by assuming $k_0^- = \lambda$, the third and fourth equations of \eqref{sys:Ir-Gray-Scott} can be reduced to:
  \begin{align}
  \begin{cases}
    \pp_t = k_2^+ \v, \\
    s^\infty_t = \lambda (k_0^+ \u - s^\infty).
  \end{cases}
  \end{align}
Notice that and $s^\infty$ can be easily represented as:
  \begin{align}
    s^\infty(t) = s^\infty_0 e^{-\lambda t} + \lambda \int_0^t k_0^+ \u(\tau) e^{-\lambda (t-\tau)} \d \tau,
  \end{align}
where $s^\infty_0$ denotes the initial data of $s^\infty$. This means that when the small parameter $\lambda$ goes to zero, it holds, formally, $s^\infty \to s^\infty_0$. By inserting this relation into the first equation of \eqref{sys:Ir-Gray-Scott}, and noticing the unknown $\pp$ is decoupled from the other equations, the irreversible system \eqref{sys:Ir-Gray-Scott} is finally reduced to the following formulation, at $O({1}/{\lambda}$) time scale:
  \begin{align}\label{sys:Ir-Gray-Scott-sf}
  \begin{cases}
    \u_t = d_u \Delta \u - k_1^+ \u (\v)^2 + k_0^+ \left( \frac{s^\infty_0}{k_0^+} - \u \right) , \\[7pt]
    \v_t = d_v \Delta \v + k_1^+ \u (\v)^2 - k_2^+ \v,
  \end{cases}
  \end{align}
with a decoupled equation for product $\pp$. So the quantity $\tfrac{s^\infty_0}{k_0^+}$ possesses a consistency with the feeding term from the external fields in the classical Gray-Scott system \eqref{eq:ori-GS}, by fixing its value $\tfrac{s^\infty_0}{k_0^+} =1$. This slow-fast dynamics presentation may provide the asymptotic relation between our irreversible system \eqref{sys:Ir-Gray-Scott} and the classical one \eqref{eq:ori-GS}. In the process, the birth-death scheme plays a crucial role. The interested readers can find similar ideas in \cite{FRE-18prl}. The rigorous justification of this limit is under preparation. We emphasize that the limiting procedures $\eps \rightarrow 0$ and $\lambda \rightarrow 0$ are not commutative, and the scale of $\lambda$ and $\eps$ may also be different.

Our goals in this paper are to establish the well-posedness of the reversible Gray-Scott-like system \eqref{sys:Re-Gray-Scott}, and to study above reversible-irreversible limit between the approximate reversible system \ref{sys:Re-Gray-Scott-eps} and the corresponding irreversible limit system \eqref{sys:Ir-Gray-Scott}.


\subsection{Main results}

Before presenting our main results, we first gather all notations and conventions used throughout this paper. We use $C$ to denote some positive constant that may take different values at different lines. 
For any $p \in [1,\infty)$, we introduce the Banach spaces $L^p$ equipped with the norms
    $| f |_{L^p} = (\int_\Omega |f|^p \d x )^{\frac{1}{p}},$
where $\Omega = \R^3 \,\text{or} \,\T^3$. Especially for $p = 2$, we use the notation $\l \cdot,\cdot \r$ to represent the inner product on the Hilbert space $L^2$.
The symbol $\nabla_x$ stands for the gradient operator and $\Delta_x$ denotes the Laplacian operator. For any multi-index $k=(k_1, k_2, k_3) \in \mathbb{N}^3$, we denote the higher order derivative operators
    $\p_x^k = \frac{ \p^{ |k| } }{ \p x_1^{k_1} \p x_2^{k_2} \p x_3^{k_3} },$
where $|k|=k_1+k_2+k_3$. We then define the Sobolev spaces $H^s$ endowed with the norms
  \begin{align*}
     | \, \cdot \, |_{H^s} = ( \sum_{|k| = 0}^s | \p^k_x \, \cdot \, |_{L^2}^2 )^{\frac{1}{2}}.
  \end{align*}

In this paper, we mainly investigate the well-posedness of the reversible system \eqref{sys:Re-Gray-Scott} with initial data \eqref{eq:initi-ReGS}, including the local existence with large initial data and the local convergence of the asymptotic system \ref{sys:Re-Gray-Scott-eps}. Moreover, under the smallness assumption on initial data, the reversible system \eqref{sys:Re-Gray-Scott} will admit a global-in-time solution near the equilibrium.

Our main results are expressed respectively in Theorem \ref{thm:local-existence-convergence} and Theorem \ref{thm:global-exist-2} below.

\begin{theorem}[Local well-posedness and convergence towards the irreversible system]
\label{thm:local-existence-convergence}
~

Let the domain $\Omega$ be $\T^3 \,\text{or}\,\, \R^3$. Assuming the initial data 
$(u_0,v_0,p_0,q_0) \in H^1(\Omega)$, there exists some positive constant $T>0$, depending only on the initial data, such that the Cauchy problem of the reversible Gray-Scott-like system \eqref{sys:Re-Gray-Scott} admits a unique solution $(u,v,p,q) \in L^\infty(0, T; H^1 (\Omega)) \cap L^2 (0, T; H^2 (\Omega))$, which satisfies the following energy bound:
  \begin{align}\label{eq:energy-bound-local}
     \sup_{ t \in [0,T] } \abs{(u,v,p,q)}_{H^1}^2
      + \int_0^T \abs{\nabla (u,v,p,q) }^2_{H^1} \d t
    \le C,
  \end{align}
where the bound $C$ only depends on the initial data, the maximum time $T$, and the coefficients in the system.

Furthermore, let $(u^\eps, v^\eps, p^\eps, q^\eps) \in L^\infty (0,T; H^1(\T^3)) \cap L^2 (0,T; H^2(\T^3))$ be the classical solution of the asymptotic reversible system \ref{sys:Re-Gray-Scott-eps} with initial data $(u_0,v_0,p_0,q_0)$ $\in H^1(\T^3)$. 
Then we have, as $\eps \rightarrow 0$,  
  \begin{align}
    (u^\eps, v^\eps, p^\eps, q^\eps) \longrightarrow (\u, \v, \pp, \q) \quad \text{in } C (0,T; L^2(\T^3)) \cap L^2 (0,T; H^1(\T^3)),
  \end{align}
 where $(\u, \v, \pp, \q)$ is the solution of the irreversible system \eqref{sys:Ir-Gray-Scott} with the same initial data $(u_0,v_0,p_0,q_0)$.

\end{theorem}

We now state the global existence result of the reversible system \eqref{sys:Re-Gray-Scott} near the equilibrium state. For that, we firstly introduce the notion of equilibrium. We impose initial data on system \eqref{sys:Re-Gray-Scott},
  \begin{align}\label{eq:initi-ReGS}
    u(0,x) = u_0(x),\quad v(0,x) = v_0(x),\quad p(0,x) = p_0(x),\quad q(0,x) = q_0(x)\,,
  \end{align}
and when the domain $\Omega$ is finite, i.e. $|\Omega| < \infty$,
  \begin{align}\label{eq:initi-total}
    \int_\Omega(u_0 + v_0 + p_0 + q_0)\,\mathrm{d}x = Z_0>0\,.
  \end{align}
Noticing that the system \eqref{sys:Re-Gray-Scott} satisfies the constraint of conservation of atoms, i.e., $\tfrac{\d}{\d t} \int (u+v+p+q) \d x =0$, we get formally that $(u+v+p+q)(t) =Z_0$ for any $t>0$.

Now we determine the constant solutions $(\ou, \ov, \op, \oq)$ to the reversible system \eqref{sys:Re-Gray-Scott}. Putting $(\ou, \ov, \op, \oq)$ into \eqref{sys:Re-Gray-Scott}, we have:
\begin{equation}\label{equilibrium-1}
  k^+_0\ou= k^-_0 \oq\,,\quad k^+_1\ou \ov^2= k^-_1\ov^3\,,\quad k^+_2\ov= k^-_2 \op\,,
\end{equation}
with the global conservation of mass:
\begin{equation}\label{equilibrium-2}
  \ou+ \ov+\op+ \oq = \frac{Z_0}{|\Omega|}\,.
\end{equation}
Without loss of generality, we assume $|\Omega|=1$. From \eqref{equilibrium-1} and \eqref{equilibrium-2}, we obtain two types of constant solutions to the reversible system \eqref{sys:Re-Gray-Scott}: one is $(\ou, \ov, \op, \oq)$,
  \begin{align}\label{eq:equilibrium}
  \begin{cases}
    \ou = \tfrac{k_0^- k_1^- k_2^-}{K} Z_0, \quad \ov = \tfrac{k_0^- k_1^+ k_2^-}{K} Z_0, \\[5pt]
    \op = \tfrac{k_0^- k_1^+ k_2^+}{K} Z_0, \quad \oq = \tfrac{k_0^+ k_1^- k_2^-}{K}Z_0,
  \end{cases}
  \end{align}
and $K={k_0^-k_1^-k_2^- + k_0^-k_1^+k_2^- + k_0^-k_1^+ k_2^+ + k_0^+ k_1^- k_2^-}$. The other constant solution is $(\underline{u}, 0, 0, \underline{q})$ with
  \begin{align}\label{eq:equilirium-2}
    \underline{u} = \tfrac{k_0^-}{k_0^+ + k_0^-} Z_0,\quad \underline{q} = \tfrac{k_0^+}{k_0^+ + k_0^-} Z_0.
  \end{align}

We point out that, when the domain $\Omega=\mathbb{R}^3$, the constant state $(\ou, \ov, \op, \oq)$ only satisfies the system \eqref{sys:Re-Gray-Scott}, but with infinite energy since the integral is infinite. In this case, we still use the equilibrium \eqref{eq:equilibrium} and \eqref{eq:equilirium-2} which can be seen as the limiting case of the finite domain.

Now we introduce the global existence of system \eqref{sys:Re-Gray-Scott} near the equilibrium $(\ou,\ov,\op,\oq)$. Based on the local result in Theorem \ref{thm:local-existence-convergence}, under the further assumption on the smallness of initial data, the equilibrium solution $(\ou,\ov,\op,\oq)$ can be extended globally in time. The solution $(u, v, p, q)$ can be rewritten as the following perturbation form:
  \begin{align}\label{eq:perturb-ansatz-2}
    u = \ou + \widetilde{u}, \quad
    v = \ov + \widetilde{v}, \quad
    p = \op + \widetilde{p}, \quad
    q = \oq + \widetilde{q},
  \end{align}
where the perturbation $\abs{\phi}\ll 1(\phi= \widetilde{u}, \widetilde{v}, \widetilde{p}, \widetilde{q})$. Correspondingly, the perturbative system is of the following formulation:
  \begin{align}\label{sys:perturb-2}
  \begin{cases}
    \widetilde{u}_t = d_u \Delta \widetilde{u} - k_1^+ \widetilde{u} \widetilde{v}^2 + k_1^- \widetilde{v}^3 - 2k_1^+ \ov \widetilde{u}\widetilde{v} + 2k_1^- \ov \widetilde{v}^2 - k_1^+ \ov^2 \widetilde{u} + k_1^- \ov^2 \widetilde{v} - k_0^+ \widetilde{u} + k_0^- \widetilde{q}, \\
    \widetilde{v}_t = d_v \Delta \widetilde{v} + k_1^+ \widetilde{u} \widetilde{v}^2 - k_1^- \widetilde{v}^3 + 2k_1^+ \ov \widetilde{u}\widetilde{v} - 2k_1^- \ov \widetilde{v}^2 + k_1^+ \ov^2 \widetilde{u} - k_1^- \ov^2 \widetilde{v} - k_2^+ \widetilde{v} + k_2^- \widetilde{p}, \\
    \widetilde{p}_t = d_p \Delta \widetilde{p} + k_2^+ \widetilde{v} - k_2^- \widetilde{p}, \\
    \widetilde{q}_t = d_q \Delta \widetilde{q} + k_0^+ \widetilde{u} - k_0^- \widetilde{q},
  \end{cases}
  \end{align}
where the relation $k_1^+ \ou = k_1^- \ov$, $k_0^+ \ou = k_0^- \oq$,  $k_2^+ \ov = k_2^- \op$ has been used.
The corresponding initial data is
  \begin{align}\label{IC:perturb-2}
    \widetilde{u}_0 = u_0 - \ou,\quad \widetilde{v}_0 = v_0 - \ov,\quad \widetilde{p}_0 = p_0 - \op, \quad \widetilde{q}_0 = q_0 - \oq.
  \end{align}

We mention that, the (global) well-posedness of the system \eqref{sys:Re-Gray-Scott} near steady (equilibrium) state is converted equivalently to the problem of the (global) existence of the perturbative system \eqref{sys:perturb-2} near zero solution.

\begin{theorem}[Global well-posedness]\label{thm:global-exist-2}
Let the domain $\Omega$ be the whole space $\R^3$. Then there exists a small constant $\nu > 0$, such that, if
  \begin{align}\label{IC-2-global}
    E_g^{in} = \abs{(\widetilde{u}_0,\widetilde{v}_0,\widetilde{p}_0,\widetilde{q}_0)}_{H^1}^2 \leq \nu \,,
  \end{align}
the solution to the Cauchy problem of \eqref{sys:Re-Gray-Scott} constructed above can be extended globally, with a global-in-time energy bound, i.e.,
  \begin{align}\label{eq:energy-bound-global-2}
   & \sup\limits_{t\geq 0} ( k_0^+k_1^+k_2^+ \abs{ \widetilde{u} }_{H^1}^2 + k_0^+k_1^-k_2^+ \abs{ \widetilde{v} }_{H^1}^2 + k_0^+k_1^-k_2^- \abs{ \widetilde{p} }_{H^1}^2 + k_0^-k_1^+k_2^+\abs{ \widetilde{q} }_{H^1}^2)  \no\\
   & + \int_{0}^{\infty} ( d_uk_0^+k_1^+k_2^+\abs{ \nabla_x \widetilde{u} }_{H^1}^2 + d_vk_0^+k_1^-k_2^+ \abs{ \nabla_x \widetilde{v} }_{H^1}^2 + d_p k_0^+k_1^-k_2^- | \nabla_x \widetilde{p} |_{H^1}^2 + d_qk_0^-k_1^+k_2^+ | \nabla_x \widetilde{q} |_{H^1}^2  \no\\
   & + k_0^+ k_2^+ \ov^2 | k_1^+ \widetilde{u} - k_1^-  \widetilde{v} |_{H^1}^2 + k_1^+ k_2^+ | k_0^+ \widetilde{u} - k_0^- \widetilde{q} |_{H^1}^2 + k_0^+ k_1^- \ov^2 | k_2^+ \widetilde{v} - k_2^-  \widetilde{p} |_{H^1}^2 )\d t \no\\
   & \leq C_g \nu,
  \end{align}
here $C_g$ is a constant depending only on the coefficients.
\end{theorem}

Notice that we can establish a similar global-in-time existence result near equilibrium in more regular space (say, $H^s$ for any $s \ge 2$), by almost the same argument. As a result, we refer this type of solution as classical solution due to the Sobolev embedding theorems in $L^\infty$ or in more regular H{\"o}lder spaces.

Due to the smallness of classical solutions near the (positive) equilibrium, the obtained solutions will always be strictly positive, so that the calculations on concentrations $(u,v,p,q)$ like taking logarithms make sense. Precisely speaking, our classical solutions satisfy the entropy-entropy production identity \eqref{energy-dissipation law}. This ensures the consistency of the reversible Gray-Scott model \eqref{sys:Re-Gray-Scott} with the thermodynamics theory.

For the well-posedness of irreversible Gray-Scott system \eqref{sys:Ir-Gray-Scott}, we can also use the energy estimate to state its well-posedness in the space $L^\infty(0, T; H^1 (\Omega)) \cap L^2 (0, T; H^2 (\Omega))$. There are many well-posedness results about the irreversible Gray-Scott model (\cite{KWW-05aml,Wei-99nonl,Wei-01physD}).

\subsection{Main ideas and difficulties}

We now sketch the main ideas of proving the above theorems. When we construct the local convergence result to the system \ref{sys:Re-Gray-Scott-eps}, our main goal is to derive the uniform energy estimates of the system \ref{sys:Re-Gray-Scott-eps} according to the energy estimate in the previous existence Theorem. Based on the local-in-time eneregy estimate uniformly in $\eps\in (0,1)$, we take the limit from the system \ref{sys:Re-Gray-Scott-eps} to the irreversible system \eqref{sys:Ir-Gray-Scott} as $\eps\rightarrow 0$. We mainly employ the Aubin-Lions-Simon's Theorem to obtain enough compactness such that the limits valid.

The second main result of this paper is to prove the global well-posedness result to the system \eqref{sys:Re-Gray-Scott} with small initial data around the equilibrium $(\ou, \ov, \op, \oq)$. The key point is that the norm $\abs{k_1^+ \widetilde{u} - k_1^- \widetilde{v}}_{L^2}^2$, $\abs{k_0^+ \widetilde{u} - k_0^- \widetilde{q}}_{L^2}^2$ and $\abs{k_2^+ \widetilde{v} - k_2^- \widetilde{p}}_{L^2}^2$ constructing in the dissipative term. Due to the linear terms in the perturbation system near $(\ou, \ov, \op, \oq)$  can not absorbed by the dissipative term in the left-hand side, more precisely, the first $\widetilde{u}$-equation in perturbation system \eqref{sys:perturb-2} is
  \begin{align*}
    \widetilde{u}_t = d_u \Delta \widetilde{u} - k_1^+ \widetilde{u} \widetilde{v}^2 + k_1^- \widetilde{v}^3 - 2k_1^+ \ov \widetilde{u}\widetilde{v} + 2k_1^- \ov \widetilde{v}^2 - k_1^+ \ov^2 \widetilde{u} + k_1^- \ov^2 \widetilde{v} - k_0^+ \widetilde{u} + k_0^- \widetilde{q},
  \end{align*}
in which contains the linear terms $k_0^+ \widetilde{u}$ and $k_0^- \widetilde{q}$. The treatment of these two linear terms is crucial in proving the $L^2$ estimate of proving the global well-posedness result. It is a key observation that we can combine all the linear terms in the perturbation system of \eqref{sys:Re-Gray-Scott}. We multiply by some coefficients to construct the two perfect square expression $\abs{k_0^+ \widetilde{u} - k_0^- \widetilde{q}}^2_{L^2}$ and $\abs{k_2^+ \widetilde{v} - k_0^- \widetilde{p}}^2_{L^2}$, which can be designed as dissipation term.  Generally speaking, in order to prove the global well-posedness with small initial data, one often should obtain the following type energy inequality
  \begin{align*}
    \tfrac{\d }{\d t} E_g(t) + D_g (t) \leq P(E_g(t))D_g(t).
  \end{align*}
Thus the term $P(E_g(t))D_g(t)$ can be absorbed by the diffusion term $D_g (t)$ due to the small assumption on the initial data. Moreover, besides the linear term $k_0^+ \widetilde{u}$, $k_2^+ \widetilde{v}$, $k_0^- \widetilde{q}$, $k_2^- \widetilde{p}$, there are two linear terms $k_1^+\ov^2 \widetilde{u}$, $ k_1^- \ov^2 \widetilde{v}$, coming from the perturbation $u = \ou + \widetilde{u}$ and $v = \ov + \widetilde{v}$. The main obstacle to prove the global existence results comes from the linear term in the $L^2$ estimate, our novelty is to seek an extra elimination relationship to overcome this difficulty. To be more precise, we design the dissipative term $\abs{k_1^+ \widetilde{u} - k_1^- \widetilde{v}}_{L^2}^2$ constructed by the linear terms $k_1^+\ov^2 \widetilde{u}$, $ k_1^- \ov^2 \widetilde{v}$. According to the chemical relation $k_1^+ \ou = k_1^- \ov$, the other terms of the system \eqref{sys:Re-Gray-Scott} around the equilibrium state $(\ou, \ov, \op, \oq)$ can be constructed as $P(E_g(t))D_g(t)$ in the $L^2$ estimate, i.e.,
  \begin{align*}
    \l -2k_1^+\ov \widetilde{u} \widetilde{v} + 2k_1^-\ov \widetilde{v}^2, \widetilde{u} \r_{L^2} \leq 2\ov \abs{\widetilde{v}}_{L^3}\abs{\widetilde{u}}_{L^6}\abs{k_1^+ \widetilde{u} - k_1^- \widetilde{v}}_{L^2},\\
    \l 2k_1^+\ov \widetilde{u} \widetilde{v} - 2k_1^-\ov \widetilde{v}^2, \widetilde{v} \r_{L^2} \leq 2\ov \abs{\widetilde{v}}_{L^3}\abs{\widetilde{v}}_{L^6}\abs{k_1^+ \widetilde{u} - k_1^- \widetilde{v}}_{L^2}.
  \end{align*}
Combine the Sobolev inequality and the small initial assumption, the right-hand term can be absorbed by the diffusion term. The above difficulties will vanish when we prove the global existence with small size of initial data. Then we derive the following energy inequality
  \begin{align*}
    \tfrac{\d }{\d t} E_g(t) + D_g (t) \leq (1 + E_g^{\frac{1}{2}}(t))E_g^{\frac{1}{2}}(t)D_g(t).
  \end{align*}
Based on the continuity arguments, one thereby verify the global well-posedness.

\subsection{Organizations of current paper}

The rest of this paper is as follows: a formal derivation of this four-species reaction-diffusion system of reversible Gray-Scott type model \eqref{sys:Re-Gray-Scott} will be given, by using the EnVarA, in the following section \S \ref{sec:derivation_ReGS}, containing the derivations for the mechanical and chemical reaction part.

In the next section \S \ref{Sec: well-posedness to the Re-GS-1 model}, we prove the local well-posedness and the local convergence of the reversible Gray-Scott system \eqref{sys:Re-Gray-Scott}. We first derive the a priori estimates in Lemma \ref{lemma local-Re-GS-1}. Then, based on the a priori estimates, we prove the large local solution by continuity arguments. Furthermore, we obtain the uniform bound energy estimate and derive the local convergence from \ref{sys:Re-Gray-Scott-eps} to \eqref{sys:Ir-Gray-Scott} by using the Aubin-Lions-Simon's Theorem.

In Section \S \ref{Sec: global-wellposedness to the Re-GS-2 model}, based on the local existence in the Theorem \ref{thm:local-existence-convergence} and the assumption on the smallness of initial data, the local exsitence can be extended globally in time.
\section{Derivation of the Reversible Gray-Scott-like Model} 
\label{sec:derivation_ReGS}

Our aim in this section is to derive the model \eqref{sys:Re-Gray-Scott} by using the \emph{energetic variational approach}. Two main ingredients are included in this approach: the \emph{least action principle} and the \emph{maximum dissipation principle}, which derive the conservative force and the dissipative force respectively, and force balance relation will lead to the final PDE system.  We split the deviation procedure into two steps: spatial diffusion part and chemical reaction part.

Note firstly that system \eqref{sys:Re-Gray-Scott} satisfies the following kinematics:
  \begin{align}\label{eq:kinematics-x}
  \begin{cases}
    \p_t u + \div_x (u \textbf{u}_u) = -r_1 - r_0, \\
    \p_t v + \div_x (v \textbf{u}_v) = r_1 - r_2, \\
    \p_t p + \div_x (p \textbf{u}_p) = r_2, \\
    \p_t q + \div_x (q \textbf{u}_q) = r_0,
  \end{cases}
  \end{align}
where $r_i$'s ($i=0,\,1,\,2$) are respectively the total reaction rates of the chemical reaction associating to the reversible chemical reaction process of cubic autocatalysis \eqref{eq:cubic-cata-re}. They are defined, as mentioned before, through the law of mass action:
  \begin{align}\label{eq:rates}
    r_0 = k_0^+ u - k_0^- q, \quad
    r_1 = k_1^+ u v^2 - k_1^- v^3, \quad
    r_2 = k_2^+ v - k_2^- p.
  \end{align}
We also point out that $\textbf{u}_\alpha$ is the induced velocity by the diffusion process of each species $\alpha=u,\,v,\,p,\,q$.

The energy-dissipation law obeyed by the reversible chemical reaction scheme \eqref{eq:cubic-cata-re} can be formulated as
  \begin{align}
    \frac{\d}{\d t} \mathcal{F} (u,\,v,\,p,\,q) = - \mathcal{D}_{d} - \mathcal{D}_{r},
  \end{align}
where the free energy $\mathcal{F}$ takes an entropy formulation:
  \begin{align}
    \mathcal{F} (u,\,v,\,p,\,q)
    = \int_{\Omega} \left[ u (\ln \tfrac{u}{\bar u} -1 ) + v (\ln \tfrac{v}{\bar v} -1 )
                  + p (\ln \tfrac{p}{\bar p} -1 ) + q (\ln \tfrac{q}{\bar q} -1 ) \right]  \d x,
  \end{align}
and the dissipation consists of two types of contributions: one comes from spatial diffusion and the other from chemical reaction, represented by $\mathcal{D}_{d}$ and $\mathcal{D}_{r}$ respectively,
  \begin{align*}
    \mathcal{D}_d &= \int_{\Omega} \left(\tfrac{d_u}{u} \abs{\nabla u}^2 + \tfrac{d_v}{v} \abs{\nabla v}^2 + \tfrac{d_p}{p} \abs{\nabla p}^2 + \tfrac{d_q}{q} \abs{\nabla q}^2 \right) \d x, \\
    \mathcal{D}_r &= \int_{\Omega} (k_0^+ u - k_0^- q) \ln \tfrac{k_0^+ u}{k_0^- q} \d x
     + \int_{\Omega} (k_1^+ u v^2 - k_1^- v^3) \ln \tfrac{k_1^+ u}{k_1^- v} \d x + \int_{\Omega} (k_2^+ v - k_2^- p) \ln \tfrac{k_2^+ v}{k_2^- p} \d x.
  \end{align*}

\subsection{EnVarA for the diffusion part} 
\label{sub:envara_for_diffusion_part}


Associated to the velocity $\textbf{u}_\alpha$, we can define flow map $x(X,\,t):\, \Omega \to \Omega$, in which $X$ denotes the Lagrangian coordinates and $x$ denotes Eulerian coordinates, by the following ordinary differential equation (ODE):
  \begin{align}
  \begin{cases}
    \frac{\d}{\d t} x(X,\, t) = \textbf{u}_\alpha(x(X,\, t),\,t), \\[2pt]
    x(X,\, 0) = X.
  \end{cases}
  \end{align}
We also define the deformation gradient by
  \begin{align}
    F(X,t)= \frac{\p x}{\p X}.
  \end{align}

For the sake of exposition, we denote $\textbf{c}=(c_u,\,c_v,\,c_p,\,c_q)=(u,\,v,\,p,\,q)$, then
  \begin{align}\label{eq:energy}
    \mathcal{F} (\textbf{c}) = \mathcal{F}  (u,\,v,\,p,\,q)
    = \int_{\Omega^t} \sum_{\alpha} c_{\alpha} (\ln \frac{c_{\alpha}}{\bar c_{\alpha}} - 1) \d x,
  \end{align}
so we can infer from the least action principle (LAP) that
  \begin{align}
    \delta_x \int_0^T \mathcal{F} (\textbf{c}) \d t
    & = \delta_x \int_0^T \int_{\Omega^0} \sum_{\alpha} \frac{c_{\alpha}^0}{\det F} \left( \ln \frac{c_{\alpha}^0}{\det F} - \ln \bar c_{\alpha} -1 \right) \det F \d X \d t \\\no
    & = \int_0^T \int_{\Omega^0} c_{\alpha}^0 \left( - \frac{\det F}{c_{\alpha}^0} \cdot \frac{c_{\alpha}^0}{(\det F)^2} \cdot \det F\, \mathrm{tr} (\frac{\p X}{\p x} \frac{\p \delta x}{\p X}) \right) \d X \d t \\\no
    & = - \int_0^T \int_{\Omega^0} c_{\alpha}^0 \, \mathrm{tr} (\frac{\p \delta x}{\p x}) \d X \d t \\\no
    & = - \int_0^T \int_{\Omega^t} c_{\alpha} \nabla_x \cdot (\delta x) \d x \d t \\\no
    & = \int_0^T \int_{\Omega^t} \nabla_x c_{\alpha} \cdot \delta x \d x \d t,
  \end{align}
where we have used the matrix equality $\tfrac{\d}{\d x} \det A(x) = \det A \cdot \mathrm{tr}(A^{-1} \tfrac{\d}{\d x}A)$ in the second line. We thus get
  \begin{align}
    \delta_x \int_0^T \mathcal{F} (\textbf{c}) \d t = \skp{\nabla_x c_{\alpha}}{\delta x}_{L^2_{t,x}}.
  \end{align}

We next turn to consider the dissipation $\mathcal{D}_d$ due to the diffusion contribution, which often takes the form of
  \begin{align}
    \mathcal{D}_d = \int_{\Omega} \sum_{\alpha} \tfrac{c_{\alpha}}{d_{\alpha}} \abs{\textbf{u}_{\alpha}}^2 \d x.
  \end{align}
The maximum dissipation principle (MDP) gives that
  \begin{align}
    \delta_{\textbf{u}_{\alpha}} \mathcal{D}_d = \skp{ \tfrac{c_{\alpha}}{d_{\alpha}} \textbf{u}_{\alpha} }{\delta {\textbf{u}_{\alpha}}}_{L^2_{x}}.
  \end{align}

We can get by the force balance $L^2_{t,x-} \tfrac{\delta \int \mathcal{E} \d t}{\delta x} + L^2_{x-} \tfrac{\delta \mathcal{D}}{\delta \dot x} =0$ that,
  \begin{align}
    \nabla_x c_{\alpha} + \tfrac{c_{\alpha}}{d_{\alpha}} \textbf{u}_{\alpha} = 0,
  \end{align}
namely,
  \begin{align}
    \textbf{u}_{\alpha} = -d_{\alpha} \nabla_x \ln c_{\alpha}.
  \end{align}
Inserting this into the kinematic relation \eqref{eq:kinematics-x} yields finally the exact expression of diffusion.

\subsection{EnVarA for the reaction part} 
\label{sub:envara_for_reaction_part}


As explained in the last section, a notion of general EnVarA is needed to deal with the chemical reactions. For that, noticing the conservation constraints of elements satisfied by the chemical reaction process \eqref{eq:cubic-cata-re}, that
  \begin{align}
    \frac{\d}{\d t} (u + v + p + q)(t) = 0,
  \end{align}
we can employ the so-called reaction trajectories for the above reaction process, $R_i(t)$'s with $i=0,1,2$, satisfying
  \begin{align}
    \frac{\d}{\d t} R_i(t) = \dot R_i(t) = r_i.
  \end{align}
Therefore, the concentrations of each species $c_{\alpha}$ (with $\alpha=u,\,v,\,p,\,q$) can be expressed as
  \begin{align}
  \begin{cases}
    u(t) = u_0 - R_1(t) - R_0(t), \\
    v(t) = v_0 + R_1(t) - R_2(t), \\
    p(t) = p_0 + R_2(t), \\
    q(t) = q_0 + R_0(t).
  \end{cases}
  \end{align}
The above relations may be regarded as the kinematics for the chemical reaction \eqref{eq:cubic-cata-re}.

The new state variable of reaction trajectory enables us to rewrite the free energy $\mathcal{F}$ by virtue of $\textbf{R}= (R_0,\,R_1,\,R_2)(t)$, i.e.,
  \begin{align}\label{eq:energy-R}
    \mathcal{F} (\textbf{R}) = \mathcal{F} (\textbf{c(R)}).
  \end{align}
Meanwhile, we also write the reaction dissipation part as $\mathcal{D}_r = \mathcal{D}_r (\textbf{R},\, \dot{\textbf{R}})$. So the energy-dissipation law for a pure chemical reaction process can be rewritten as
  \begin{align}\label{eq:EnDis-law-R}
    \frac{\d}{\d t} \mathcal{F} (\textbf{R}) = - \mathcal{D}_r (\textbf{R},\, \dot{\textbf{R}}).
  \end{align}

Since chemical reactions are often far away from equilibrium states, the reaction dissipation may not be quadratic with respect to $\dot{\textbf{R}}$. This fact is different from the quadratic dissipation functional in a mechanical system, which is also the reason that we need the notion of generalized EnVarA.

We assume that the nonnegative reaction dissipation $\mathcal{D}_r$ takes the form of
  \begin{align}
    \mathcal{D}_r (\textbf{R},\, \dot{\textbf{R}}) = \skp{\mathcal{G}_r(\textbf{R},\, \dot{\textbf{R}})}{\dot{\textbf{R}}},
  \end{align}
this, combining with the fact $\tfrac{\d}{\d t} \mathcal{F} (\textbf{R}) = \skp{\tfrac{\delta \mathcal{F}}{\delta \textbf{R}}}{\dot{\textbf{R}}}$, yields that,
  \begin{align}\label{eq:grad-flow-gnrl}
    \mathcal{G}_r (\textbf{R}, \dot{\textbf{R}}) = - \frac{\delta \mathcal{F}}{\delta \textbf{R}}.
  \end{align}
We refer this as a general gradient flow. The exact expression of $r_i$'s will be revisited by choosing
  \begin{align}
    \mathcal{D}_{r_0} (R_0,\, \dot R_0) & = \dot R_0 \ln \left( \tfrac{\dot R_0}{k_0^- q} +1\right), \\
    \mathcal{D}_{r_1} (R_1,\, \dot R_1) & = \dot R_1 \ln \left( \tfrac{\dot R_1}{k_1^- v^3} +1\right), \\
    \mathcal{D}_{r_2} (R_2,\, \dot R_2) & = \dot R_2 \ln \left( \tfrac{\dot R_2}{k_2^- p} +1\right).
  \end{align}
Indeed, for $i=0$, direct calculations imply that
  \begin{align}
    \frac{\delta \mathcal{F}}{\delta R_0}
    = \sum_{\alpha} \ln \frac{c_{\alpha}}{\bar c_{\alpha}} \cdot \frac{\p c_{\alpha}}{\p R_0}
    = - \ln \frac{u}{\bar u} + \ln \frac{q}{\bar q}
    = - \ln \frac{k_0^+ u}{k_0^- q},
  \end{align}
where we have used the equilibrium equality $k_0^+ \bar u = k_0^- \bar q$. As a result, the above general gradient flow \eqref{eq:grad-flow-gnrl} leads to the following relation:
  \begin{align}
    r_0 = \dot R_0 = k_0^+ u - k_0^- q,
  \end{align}
which is exactly the same formulation as Eq. \eqref{eq:rates} given by the law of mass action. The other two rates $r_1$ and $r_2$ can be derived by a similar process. Inserting these expressions into the kinematic relation \eqref{eq:kinematics-x} yields the expression of reaction part.

Finally, combining diffusion contribution and reaction contribution enables us to get the partial differential equations (PDEs) governing the reversible chemical reaction process of cubic autocatalysis, namely, system \eqref{sys:Re-Gray-Scott}. Note that this is a four-species reaction-diffusion system of reversible Gray-Scott type.

\section{Local Well-posedness and Convergence Towards the Irreversible System}\label{Sec: well-posedness to the Re-GS-1 model}

In this section, we will employ the energy method to prove the local in time existence of the system \eqref{sys:Re-Gray-Scott} with large initial data. Moreover, according to the energy bound in the existence result, we obtain the uniform energy bound of $(u^\eps, v^\eps, p^\eps, q^\eps)$ and $(u^\eps_t, v^\eps_t, p^\eps_t, q^\eps_t)$. Combining with the Aubin-Lions-Simon's Theorem, we derive the reversible-irreversible limit in local time.

\subsection{A priori estimate to the reversible Gray-Scott system}\label{A_Priori-local-1}
In this subsection, the a priori estimate of the system \eqref{sys:Re-Gray-Scott} will be accurately derived from employing the energy method. We now introduce the following energy functional $E_{L} (t)$ and energy dissipative rate functional $D_{L} (t)$:
  \begin{align}\label{Loc-Re-Energy-1}
    E_{L} (t) & = \abs{u}^2_{H^1} + \abs{v}^2_{H^1} + \abs{p}^2_{H^1} + \abs{q}^2_{H^1} \,,  \no\\
    D_{L} (t) & = \tfrac{d_u}{2}\abs{\nabla u }^2_{H^1}  + \tfrac{d_v}{2}\abs{\nabla v }^2_{H^1} +  d_p \abs{\nabla p }^2_{H^1} +  d_q \abs{\nabla q }^2_{H^1} + k_1^- \abs{v^2}_{H^1}^2 + k_1^+ \abs{uv}_{L^2}^2 \no\\
     &\quad  + k_1^+ \abs{\p u\cdot v }_{L^2}^2 + k_0^+\abs{u}_{H^1}^2  + k_2^+\abs{v}_{H^1}^2 + k_2^-\abs{p}_{H^1}^2 + k_0^-\abs{q}_{H^1}^2 \,.
  \end{align}
Now we state the priori estimate as follows:

  \begin{lemma}\label{lemma local-Re-GS-1}
    Assume that $(u(t,x), v(t,x), p(t,x), q(t,x))$ is a sufficiently smooth solution to system \eqref{sys:Re-Gray-Scott} on the interval $[0,T]$. Then there is a positive constant $C_L = C_L(d_u, d_v, d_p,$\, $d_q, k_0^+, k_0^-, k_1^+, k_1^-, k_2^+, k_2^-)>0$, such that
  \begin{align*}
    \tfrac{\d}{\d t} E_L (t) + D_L (t) \leq C_L ( E_L (t) + E_L^3 (t))
  \end{align*}
holds for all $t\in[0,T]$.
  \end{lemma}

\begin{proof}[Proof of Lemma \ref{lemma local-Re-GS-1}]
We first derive the $L^2$-estimate, which will contain the major structures of the energy functionals. Then we estimate the higher order energy bound, which shall be consistent with the structures of $L^2$-estimate.

\bigskip\noindent\textbf{Step 1. $L^2$ estimates.}

We take $L^2$-inner product with $u(t,x)$ on the first equation of system \eqref{sys:Re-Gray-Scott}, and integrate by parts over $x\in \Omega$ . We thereby have
  \begin{align}\label{Re-GS-L2-u}
     \tfrac{1}{2} \tfrac{\d}{\d t} \abs{u}_{L^2}^2 + d_u \abs{\nabla u}_{L^2}^2 + k_1^+ \abs{uv}_{L^2}^2 + k_0^+ \abs{u}_{L^2}^2 \leq  k_1^{-}\abs{v}_{L^4}^3\abs{u}_{L^4} + k_0^- \abs{u}_{L^2}\abs{ q}_{L^2}.
  \end{align}
Based on the integration by parts, the H\"{o}lder inequality, from taking $L^2$-inner product with $v(t,x)$ on the second equation of system \eqref{sys:Re-Gray-Scott}, integrating by parts over $x\in \Omega$, we derive
  \begin{align}\label{Re-GS-L2-v}
   \tfrac{1}{2} \tfrac{\d}{\d t} \abs{v}_{L^2}^2 + d_v | \nabla v |_{L^2}^2 + k_1^{-}\abs{v^2}_{L^2}^2 + k_2^+ \abs{ v}_{L^2}^2 \leq k_1^+\abs{u}_{L^4} \abs{v}_{L^4}^3 +  k_2^- \abs{ v}_{L^2}\abs{ p}_{L^2} .
  \end{align}
We then multiply the third equation of system \eqref{sys:Re-Gray-Scott} by $p(t,x)$, and integrate by parts over $x\in \Omega$. We thereby obtain
  \begin{align}\label{Re-GS-L2-p}
    \tfrac{1}{2} \tfrac{\d}{\d t} \abs{p}_{L^2}^2 + d_p \abs{\nabla p }_{L^2}^2 + k_2^- \abs{ p}_{L^2}^2 \leq  k_2^+ \abs{ v}_{L^2}\abs{ p}_{L^2} .
  \end{align}
By taking $L^2$-inner product with $q(t,x)$ on the forth equation of system \eqref{sys:Re-Gray-Scott}, and integrating by parts over $x\in \Omega$, we have
  \begin{align}\label{Re-GS-L2-q}
    \tfrac{1}{2} \tfrac{\d}{\d t} \abs{q}_{L^2}^2 + d_q \abs{\nabla q }_{L^2}^2 + k_0^- \abs{ q}_{L^2}^2  \leq  k_0^+ \abs{ u}_{L^2}\abs{ q}_{L^2}.
  \end{align}
Adding the inequalities \eqref{Re-GS-L2-u}, \eqref{Re-GS-L2-v}, \eqref{Re-GS-L2-p} and \eqref{Re-GS-L2-q}, and combining the Sobolev embedding $H^1(\Omega)\hookrightarrow L^4(\Omega)$, we see that
  \begin{align}\label{Re-GS-L2}
    & \tfrac{1}{2} \tfrac{\d}{\d t} ( \abs{u}_{L^2}^2 + \abs{v}_{L^2}^2 +  \abs{p}_{L^2}^2 +  \abs{q}_{L^2}^2) + d_u \abs{\nabla u }_{L^2}^2 + d_v \abs{\nabla v }_{L^2}^2 +  d_p\abs{\nabla p }_{L^2}^2 + d_q  \abs{\nabla q }_{L^2}^2 + k_1^+ \abs{uv}_{L^2}^2 \no\\
    & \quad + k_1^-  \abs{v^2}_{L^2}^2 + k_0^+ \abs{u}_{L^2}^2 + k_2^+ \abs{v}_{L^2}^2 + k_2^- \abs{p}_{L^2}^2 + k_0^- \abs{q}_{L^2}^2  \no\\
    & \leq k_1^- \abs{v}_{H^1}^3 \abs{u}_{H^1} + k_1^+ \abs{u}_{H^1}\abs{v}_{H^1}^3 + k_0^- \abs{q}_{L^2}\abs{u}_{L^2} + k_2^- \abs{v}_{L^2}\abs{p}_{L^2} + k_2^+ \abs{v}_{L^2}\abs{p}_{L^2} \no\\
    & \quad + k_0^+ \abs{q}_{L^2}\abs{u}_{L^2}.
  \end{align}

\bigskip\noindent\textbf{Step 2. $H^1$ estimates.}

First, we act derivative on the first equation of system $\eqref{sys:Re-Gray-Scott}$, take $L^2$-inner product by dot with $\p_x u$ and integrate by parts over $x \in \Omega$. We thereby have
  \begin{align}\label{Re-GS-H1-u}
    \tfrac{1}{2}& \tfrac{\d}{\d t} \abs{\p u}_{L^2}^2 + d_u | \p  \nabla u |_{L^2}^2 + k_1^+ \abs{\p u\cdot v }_{L^2}^2 + k_0^+\abs{\p u}_{L^2}^2 \no\\
                & \leq  2 k_1^+ \abs{u}_{L^6} \abs{v}_{L^6} \abs{\p u}_{L^6} \abs{\p v}_{L^2} + 3 k_1^- \abs{v}_{L^6}^2 \abs{\p u}_{L^6} \abs{\p v}_{L^2} + k_0^- \abs{\p q}_{L^2} \abs{\p u}_{L^2} \no\\
                & \leq \tfrac{d_u}{2} \abs{\nabla u}_{H^1}^2 + \tfrac{4(k_1^+)^2}{d_u} \abs{u}_{H^1}^2\abs{v}_{H^1}^4  + \tfrac{9(k_1^-)^2}{d_u} \abs{v}_{H^1}^6 + k_0^- \abs{\p q}_{L^2} \abs{\p u}_{L^2}.
  \end{align}

For the second equation of system $\eqref{sys:Re-Gray-Scott}$, act derivative and take $L^2$-inner product by dot with $\p_x v$ and integrate by parts over $x\in \Omega$ . We obtain
  \begin{align}\label{Re-GS-H1-v}
    \tfrac{1}{2} & \tfrac{\d}{\d t} \abs{\p v}_{L^2}^2 + d_v | \p \nabla v|_{L^2}^2  + 3k_1^- | v \cdot \p v|_{L^2}^2 + k_2^+ \abs{\p v}_{L^2}^2 \no\\
                 & \leq  k_1^+ \abs{v}_{L^6}^2 \abs{\p u}_{L^2}\abs{\p v}_{L^6} + 2 k_1^+ \abs{u}_{L^6} |  v |_{L^6} \abs{\p v}_{L^6}\abs{\p v}_{L^2} + k_2^- \abs{\p p}_{L^2} \abs{\p v}_{L^2}\no\\
                 & \leq \tfrac{d_v}{2} \abs{\nabla v}_{H^1}^2 + \tfrac{(k_1^+)^2}{d_v} \abs{u}_{H^1}^2\abs{v}_{H^1}^4  + \tfrac{4(k_1^+)^2}{d_v} \abs{u}_{H^1}^2 \abs{v}_{H^1}^4 + k_2^- \abs{\p p}_{L^2} \abs{\p v}_{L^2}.
  \end{align}

We next act derivative on the third equation of system $\eqref{sys:Re-Gray-Scott}$, take $L^2$-inner product by dot with $\p_x p$ and integrate by parts over $x\in \Omega$. Then we have,
  \begin{align}\label{Re-GS-H1-p}
    \tfrac{1}{2} \tfrac{\d}{\d t} | \p p |_{L^2}^2 + d_p \abs{\p \nabla p}_{L^2}^2 + k_2^- \abs{\p p}_{L^2}^2  \leq  k_2^+ \abs{\p v}_{L^2} \abs{\p p}_{L^2} .
  \end{align}

Next, from acting derivative on the forth equation of system \eqref{sys:Re-Gray-Scott}, taking $L^2$-inner product by dot with $\p_x q$ and integrating by parts over $x\in \Omega$, we deduce that
  \begin{align}\label{Re-GS-H1-q}
    \tfrac{1}{2} \tfrac{\d}{\d t} | \p q |_{L^2}^2 + d_q \abs{\p \nabla q}_{L^2}^2 + k_0^- \abs{\p q}_{L^2}^2  \leq  k_0^+ \abs{\p u}_{L^2} \abs{\p q}_{L^2}.
  \end{align}

Due to the Sobolev embedding $H^1 (\Omega)\hookrightarrow L^4(\Omega)$ and $H^1(\Omega)\hookrightarrow L^6(\Omega)$, combine with the $L^2$ estimate, it is therefore derived from adding the inequalities \eqref{Re-GS-H1-u}, \eqref{Re-GS-H1-v}, \eqref{Re-GS-H1-p} and \eqref{Re-GS-H1-q},
  \begin{align}\label{Re-GS-local-1}
    & \tfrac{1}{2} \tfrac{\d}{\d t} ( \abs{u}_{H^1}^2 + \abs{v}_{H^1}^2 +  \abs{p}_{H^1}^2 +  \abs{q}_{H^1}^2) + \tfrac{d_u}{2}\abs{\nabla u }_{H^1}^2 + \tfrac{d_v}{2} | \nabla  v |_{H^1}^2 + d_p \abs{\nabla p }_{H^1}^2 + d_q \abs{\nabla q }_{H^1}^2 \no\\
    & \quad + k_1^- \abs{v^2}_{H^1}^2 + k_1^+ \abs{uv}_{L^2}^2 + k_1^+ | \p u \cdot v |_{L^2}^2 + k_0^+ \abs{u}_{H^1}^2 + k_2^+ \abs{v}_{H^1}^2 + k_2^- \abs{p}_{H^1}^2 + k_0^- \abs{q}_{H^1}^2  \no\\
    & \leq k_1^- \abs{v}_{H^1}^3 \abs{u}_{H^1} + k_1^+ \abs{u}_{H^1}\abs{v}_{H^1}^3 + \tfrac{4(k_1^+)^2}{d_u} \abs{u}_{H^1}^2\abs{v}_{H^1}^4  + \tfrac{9(k_1^-)^2}{d_u} \abs{v}_{H^1}^6 + \tfrac{(k_1^+)^2}{d_v} \abs{u}_{H^1}^2\abs{v}_{H^1}^4  \no\\
    & \quad  + \tfrac{4(k_1^+)^2}{d_v} \abs{u}_{H^1}^2 \abs{v}_{H^1}^4 + k_0^- \abs{q}_{H^1}\abs{u}_{H^1} + k_2^- \abs{v}_{H^1}\abs{p}_{H^1} + k_2^+ \abs{v}_{H^1}\abs{p}_{H^1} + k_0^+ \abs{q}_{H^1}\abs{u}_{H^1} \no\\
    & \leq C_L (E_L(t) + E_L^3(t)),
  \end{align}
where $C_L = k_0^- + k_0^+ + k_1^- +  k_1^+  + k_2^- + k_2^+ + \tfrac{4(k_1^+)^2}{d_u} + \tfrac{9(k_1^-)^2}{d_u} + \tfrac{(k_1^+)^2}{d_v} + \tfrac{4(k_1^+)^2}{d_v}$. Recalling the definitions of the energy functionals $E_L(t)$ and $D_L(t)$ in \eqref{Loc-Re-Energy-1}, we finish the proof of Lemma \ref{lemma local-Re-GS-1} from the inequality \eqref{Re-GS-local-1}.
\end{proof}

\subsection{Local well-posedness with large initial data }\label{Sub: Local-Result to the Re-GS-1 model}

The aim of this subsection is to prove the local well-posedness of the system \eqref{sys:Re-Gray-Scott} with large initial data, namely, prove the existence result of the Theorem \ref{thm:local-existence-convergence}. We first construct the linear approximate system by iteration. Then the key step is to prove the existence of the uniform positive time lower bound to the iterative approximate system and thereby the uniform energy bound will hold. Finally, by compactness arguments, we justify the local existence results.

\smallskip\noindent{\textbf{The iteration scheme}}:
We construct the approximate system of \eqref{sys:Re-Gray-Scott} by iteration as follows, for all integers $n \geq 0$, that
  \begin{align}\label{It-Re-Gray-Scott}
  \begin{cases}
    u_t^{n+1} = d_u \Delta u^{n+1} - k_1^+ u^{n+1} (v^n)^2 + k_1^- (v^n)^3 - k_0^+ u^{n+1} + k_0^- q^n \\
    v_t^{n+1} = d_v \Delta v^{n+1} + k_1^+ u^n v^n v^{n+1}  - k_1^- (v^n)^2 v^{n+1}- k_2^+ v^{n+1} + k_2^- p^n\\
    p_t^{n+1} = d_p \Delta p^{n+1} + k_2^+ v^n- k_2^- p^{n+1} \\
    q_t^{n+1} = d_q \Delta q^{n+1} + k_0^+ u^n - k_0^- q^{n+1}.
  \end{cases}
  \end{align}
And we start the approximating system from $n=0$ with
  \begin{align}\label{It-IC-Re-Gray-Scott-1}
    u^{in}(t,x)= u_0(x)  &, \quad v^{in}(t,x) = v_0(x), \no\\
     p^{in}(t,x)= p_0(x) &, \quad q^{in}(t,x) = q_0(x) .
  \end{align}

In the arguments proving the convergence ($n \rightarrow \infty$) of the approximate solutions \eqref{It-Re-Gray-Scott}, it is essential to obtain uniform (in $n \geq 0$) energy estimates of \eqref{It-Re-Gray-Scott} in a uniform lower bound lifespan time, whose derivations are the almost same as the derivations of the a priori estimates for the system \eqref{sys:Re-Gray-Scott}. 
The arguments of the uniform lower bound lifespan time can be referred to \cite{JL-SIAM-2019}, for instance. The convergence arguments are a standard process. For simplicity, we will only consider the a priori estimates in Lemma \ref{lemma local-Re-GS-1} for the smooth solutions of system \eqref{sys:Re-Gray-Scott} on some time interval.

\smallskip\noindent{\textbf{The maximum time of existence}}:
From Lemma \ref{lemma local-Re-GS-1}, we see that
  \begin{align}\label{Loc-1}
    \tfrac{\d}{\d t} E_{L} (t) + D_{L} (t) \leq C_L ( 1 + E_L^2 (t) ) E_{L} (t) \,,
  \end{align}
where the energy functionals $E_L (t)$ and $D_L (t)$ are defined in \eqref{Loc-Re-Energy-1}. Then \eqref{Loc-1} implies
  \begin{align}\label{Loc-2}
    \frac{d}{d t} \Big[ \ln \frac{E_L (t)}{(1 + E_L^2 (t) )^\frac{1}{2}} \Big] \leq C_L \,.
  \end{align}
Noticing that
  \begin{align*}
    E_L (0) = E_L^{ in} : = | u^{in} |^2_{H^1} + \abs{v^{in}}^2_{H^1} + \abs{p^{in}}^2_{H^1} + | q^{in} |_{H^1}^2 < \infty \,,
  \end{align*}
we derive from integrating the inequality \eqref{Loc-2} over $[0,t]$ that
  \begin{align}\label{Loc-3}
    \frac{E_L (t)}{ \big[ 1 + E_L^2 (t) \big]^\frac{1}{2} } \leq \frac{E_L^{in} }{ \big[ 1 + (E_L^{ in})^2 \big]^\frac{1}{2} } e^{C_L t} = A_1(t) \,.
  \end{align}
Consider the function
  $$H (y) = \frac{y}{(1 + y^2)^\frac{1}{2}} $$
for $y \geq 0$. It is easy to see that $H (0) = 0$, $\lim\limits_{y \rightarrow \infty} H (y) = 1$, and $H' (y)>0$ is strictly decreasing in $[0,\infty)$.  We therefore see that if $A_1(t) < 1$, the nonlinear inequality \eqref{Loc-3} on the functional $E_L (t)$ can be solved as
  \begin{align}\label{Loc-4}
    E_L (t) \leq H^{-1} (A_1(t)) = H^{-1} \left( \tfrac{E_L^{in} }{ \big[ 1 + (E_L^{in})^2 \big]^\frac{1}{2} } e^{C_L t} \right)  = B (t) \,.
  \end{align}
Notice that $B (t)$ is strictly increasing on $t \geq 0$. $A_1(t) < 1$ implies that
$$t < \tfrac{1}{C_L} \ln \tfrac{[1 + (E_L^{ in})^2]^\frac{1}{2}}{E_L^{ in}} \,. $$
Consequently, from \eqref{Loc-1} and \eqref{Loc-4}, we derive that for any $0 < T < \tfrac{1}{C_L} \ln \tfrac{[1 + (E_L^{ in})^2]^\frac{1}{2}}{E_L^{ in}} $,
  \begin{align*}
    E_L (t) + \int_0^t D_L (\tau) \d \tau \leq E_L^{ in} + C_LT B(T) ( 1 + B^2 (T) ) = A (T, E_L^{ in}, C_L)
  \end{align*}
holds for all $t \in [0,T]$. Consequently, we conclude the local existence of the system \eqref{sys:Re-Gray-Scott} with large initial data.

\subsection{Reversible-irreversible limit}\label{local-convergence-1}

We now in a position of proving the convergence result of Theorem \ref{thm:local-existence-convergence}, according to local energy bound of the existence result in Theorem \ref{thm:local-existence-convergence}, we aim at deriving the uniform estimates of system \ref{sys:Re-Gray-Scott-eps}. Based on the local energy bound \eqref{Re-GS-local-1} in Theorem \ref{thm:local-existence-convergence}, we now introduce the following energy functional $E^\eps_{L} (t)$ and energy dissipative rate functional $D^\eps_{L} (t)$ similarly:
  \begin{align*}
    E^\eps_{L} (t) & = \abs{u^\eps}^2_{H^1} + \abs{v^\eps}^2_{H^1} + \abs{p^\eps}^2_{H^1} + \abs{q^\eps}^2_{H^1} \,,  \no\\
    D^\eps_{L} (t) & = \tfrac{d_u}{2}\abs{\nabla u^\eps }^2_{H^1}  + \tfrac{d_v}{2}\abs{\nabla v^\eps }^2_{H^1} +  d_p \abs{\nabla p^\eps }^2_{H^1} +  d_q \abs{\nabla q^\eps }^2_{H^1} + k_1^- \abs{(v^\eps)^2}_{H^1}^2 + k_1^+ \abs{u^\eps v^\eps}_{L^2}^2 \no\\
     &\quad  + k_1^+ \abs{\p u^\eps \cdot v^\eps }_{L^2}^2 + k_0^+\abs{u^\eps}_{H^1}^2  + k_2^+\abs{v^\eps}_{H^1}^2 + k_2^-\abs{p^\eps}_{H^1}^2 + k_0^-\abs{q^\eps}_{H^1}^2 \,,
  \end{align*}
  and satisfy the following inequality:
  \begin{align}\label{limit-bound-1}
    \tfrac{1}{2} \tfrac{\d}{\d t} E^\eps_{L} (t) + D^\eps_{L} (t) \leq C_L^\eps(E^\eps_L(t) + (E^\eps_L(t))^3).
  \end{align}
We aim at deriving the system \eqref{sys:Ir-Gray-Scott} from the system \ref{sys:Re-Gray-Scott-eps} as $\eps \rightarrow 0$. Based on the definition of $C_L$, we know that
  \begin{align*}
    C_L^\eps \leq \overline{C}_L + 1
  \end{align*}
as $\eps \rightarrow 0$, where $\overline{C}_L = k_0^- + k_0^+ +  k_1^+ + k_2^+ + \tfrac{4(k_1^+)^2}{d_u} + \tfrac{(k_1^+)^2}{d_v} + \tfrac{4(k_1^+)^2}{d_v}$.  We obtain the uniform bound with respect to $\eps$ in the following sense, for any $0 < T < \tfrac{1}{\overline{C}_L + 1} \ln \tfrac{[1 + (E_L^{in})^2]^\frac{1}{2}}{E_L^{ in}}$, $0< \eps <1$,
  \begin{align}\label{uniform-bound-1}
    &\sup\limits_{t\in[0,T]}(\abs{u^\eps}_{H^1}^2 + \abs{v^\eps}_{H^1}^2 + \abs{p^\eps}_{H^1}^2 +  \abs{q^\eps}_{H^1}^2) \no\\
    &\quad + \int_{0}^{T} (\abs{\nabla u^\eps}_{H^1}^2 +  | \nabla v^\eps |_{H^1}^2 +  \abs{\nabla p^\eps }_{H^1}^2 + \abs{\nabla q^\eps }_{H^1}^2) \d t \leq C(T,E_L^{in},\overline{C}_L),
  \end{align}
where the constant $C$ is independent of $\eps$. Then we know that $(u^\eps, v^\eps, p^\eps, q^\eps )$ is uniformly bounded in $L^\infty (0,T; H^1) \cap L^2 (0,T; H^2)$.

In order to use the compactness Aubin-Lions-Simon's Theorem to prove the convergence from system \ref{sys:Re-Gray-Scott-eps} to system \eqref{sys:Ir-Gray-Scott}, we have to obtain the uniform estimates for the time derivative of $(u^\eps, v^\eps, p^\eps, q^\eps)$ in the following.

\smallskip\noindent{\textbf{Uniform estimates for time derivatives}}.
Firstly, according to the equation of $u^\eps$ in the system \ref{sys:Re-Gray-Scott-eps}, we obtain
\begin{align*}
  | u^\eps_t |_{L^2} \leq d_u \abs{\nabla u^\eps}_{H^1} + k_1^+ \abs{u^\eps}_{H^1}\abs{v^\eps}_{H^1}^2 + \eps | (v^\eps)|^3_{H^1} + k_0^+ \abs{u^\eps}_{L^2} + k_0^- \abs{q^\eps}_{L^2}.
\end{align*}

It follows that
  \begin{small}
  \begin{align}\label{bound-ut-L2-1}
    \int_0^T | u^\eps_t |_{L^2}^2 \d t 
    & \leq d_u^2 \int_0^T \abs{\nabla u^\eps}_{H^1}^2  \d t + (k_1^+)^2 T (\sup\limits_{ t\in[0,T]}\abs{u^\eps}_{H^1})^2(\sup\limits_{ t\in[0,T]}\abs{v^\eps}_{H^1})^4 + \eps^2 T (\sup\limits_{ t\in[0,T]}\abs{v^\eps}_{H^1})^6 \no\\
    & \quad + (k_0^+)^2 T (\sup\limits_{ t\in[0,T]}\abs{u^\eps}_{H^1})^2 + (k_0^-)^2 T (\sup\limits_{ t\in[0,T]}\abs{q^\eps}_{H^1})^2, \no\\
    & \leq C,
  \end{align}
  \end{small}
where in the last inequality we have used the uniform bound obtained in \eqref{uniform-bound-1}. Then we obtain $u^\eps_t$ is uniformly bounded in the space $L^2(0,T; L^2)$. Specifically speaking,
\begin{align*}
  | u^\eps_t |_{L^2(0,T; L^2)}\leq C(T) \,\,\text{for all}\,\, 0< \eps < 1
\end{align*}
is valid for $0 < T < \tfrac{1}{\overline{C}_L+1} \ln \tfrac{[1 + (E_L^{ in})^2]^\frac{1}{2}}{E_L^{ in}}$.

One notices that,
  \begin{align}\label{imbedding-1}
    H^2(\T^3) \hookrightarrow H^1(\T^3) \hookrightarrow L^2(\T^3)
  \end{align}
where the embedding of $H^2$ in $H^1$ is compact and the embedding of $H^1$ in $L^2$ is naturally continuous. Then from Aubin-Lions-Simon's Theorem, the bounds \eqref{uniform-bound-1}, \eqref{bound-ut-L2-1} and the embeddings \eqref{imbedding-1}, we deduce that there exists a $\u \in L^2(0,T; H^1)$ such that
  \begin{align*}
  u^\eps \rightarrow \u
  \end{align*}
strongly in $L^2(0,T; H^1)$ as $\eps \rightarrow 0$.

Moreover, note that
  \begin{align}\label{imbedding-2}
    H^1(\T^3) \hookrightarrow L^2(\T^3) \hookrightarrow L^2(\T^3)
  \end{align}
with the compact embedding of $H^1$ in $L^2$. Then from the Aubin-Lions-Simon's Theorem, the bound \eqref{uniform-bound-1}, \eqref{bound-ut-L2-1} and the embeddings \eqref{imbedding-2}, we deduce that the sequence $u^\eps$ strongly convergent to $\u$ in $C(0,T;L^2)$.

We obtain that the sequence $u^\eps_t$ is uniformly bounded in the space $L^2(0,T; L^2)$ and the strong convergence of the sequence $u^\eps$ in $C(0,T; L^2) \cap L^2(0,T; H^1)$. Similarly, we can obtain the same results of the sequence $v^\eps, p^\eps, q^\eps$, and we omitted it here.

Next we have to prove that the limit $(\u, \v, \pp, \q)$ is the solution of system \eqref{sys:Ir-Gray-Scott}.

\smallskip\noindent{\textbf{Convergence}}.
%
%
In the space $L^\infty(0,T; L^2)$, due to the Sobolev imbedding $H^1\hookrightarrow L^6$, we have the uniformly bounded for the $(v^\eps)^3$ and $p^\eps$ of the system \ref{sys:Re-Gray-Scott-eps}
  \begin{align}\label{v3-p-L2}
    \begin{cases}
    | (v^\eps)^3 |_{L^\infty(0,T; L^2)} \leq \abs{v^\eps}_{L^\infty(0,T; L^6)}^3 \leq \abs{v^\eps}_{L^\infty(0,T; H^1)}^3 \leq C,\\
    | p^\eps |_{L^\infty(0,T; L^2)} \leq \abs{p^\eps}_{L^\infty(0,T; H^1)} \leq C,
    \end{cases}
  \end{align}
where we use the uniform energy bound \eqref{uniform-bound-1}.

Then, in the space $L^2 (0,T; H^1)$, we obtain
  \begin{align}\label{v3-H1}
    |  (v^\eps)^3 |^2_{L^2 (0,T; H^1)} &=  \int_0^T | (v^\eps)^3 |^2_{H^1}  \d t \no\\
    & \leq  \int_{0}^{T} \abs{(v^\eps)^3}^2_{L^2}  \d t + 3 \int_{0}^{T} \abs{(v^\eps)^2\cdot \nabla v^\eps}_{L^2}^2  \d t \no\\
    & \leq T (\sup\limits_{t\in[0,T]}\abs{v^\eps}_{H^1})^6  + 3 (\sup\limits_{t\in[0,T]}\abs{v^\eps}_{H^1})^4 \int_{0}^{T} |\nabla v^\eps |_{H^1}^2 \d t \no\\
    & \leq C
  \end{align}
and
  \begin{align}\label{p-H1}
    | p^\eps |^2_{L^2 (0,T; H^1)} &\leq \int_0^T \abs{p^\eps}^2_{H^1}  \d t \leq  T(\sup\limits_{t\in[0,T]}\abs{p^\eps}_{H^1})^2 \leq C.
  \end{align}
Because of the  fact that the solution $(u^\eps, v^\eps, p^\eps, q^\eps)$ of system \eqref{sys:Re-Gray-Scott} is uniformly bounded in the space $L^\infty (0,T; H^1) \cap L^2 (0,T; H^2)$,  we know that $(v^\eps)^3$ and $p^\eps$ are bounded in $L^2 (0,T; H^1)$. Then it follows that, taking formally $\eps \rightarrow 0$,
  \begin{align*}
    \eps (v^\eps)^3 \rightarrow 0 \,\, \text{in}\,\, L^\infty (0,T; L^2) \cap L^2 (0,T; H^1),\\
    \eps p^\eps \rightarrow 0 \,\, \text{in}\,\, L^\infty (0,T; L^2) \cap L^2 (0,T; H^1).
  \end{align*}
Moreover, according to the strong convergence result and  Fatou's Lemma, we derive that
 \begin{align}\label{conver-uv2}
   & \abs{u^\eps (v^\eps)^2 - \u (\v)^2}_{L^2 (0,T; L^2)} \no\\
   & = \abs{(u^\eps - \u)(v^\eps)^2 + \u v^\eps (v^\eps - \v) + \u \v (v^\eps - \v) }_{L^2 (0,T; L^2)} \no\\
   & \leq \sup \abs{v^\eps }_{H^1}^4 \int_{0}^{T} \abs {u^\eps - \u}^2_{H^1} \d t + \sup \abs{\u }_{H^1}^2\sup \abs{v^\eps }_{H^1}^2 \int_{0}^{T} \abs {v^\eps - \v}^2_{H^1} \d t  \no\\
   &\quad + \sup \abs{\u}_{H^1}^2 \sup \abs{\v }_{H^1}^2 \int_{0}^{T} \abs {v^\eps - \v}_{H^1}^2 \d t  \no\\
   & \rightarrow 0, \quad \text{ as}\quad \eps \rightarrow 0.
 \end{align}
Hence, taking formally $\eps \rightarrow 0$ in system \ref{sys:Re-Gray-Scott-eps}, one can obtain the irreversible Gray-Scott system \eqref{sys:Ir-Gray-Scott}. We conclude $(\u, \v, \pp, \q)\in C(0,T; L^2) \cap L^2 (0,T; H^1)$ is indeed a solution of system \eqref{sys:Ir-Gray-Scott}. This completes the convergence result of Theorem \ref{thm:local-existence-convergence}.

\section{Global Well-posedness of the System \eqref{sys:Re-Gray-Scott} near Equilibrium}\label{Sec: global-wellposedness to the Re-GS-2 model}

In this section, we obtain the global in time existence under small size of initial data near the equilibrium state $(\ou, \ov, \op, \oq)$ of system \eqref{sys:Re-Gray-Scott}. We now introduce the following global energy functional $E_g(t)$ and global energy dissipative rate functional $D_g(t)$:
  \begin{align}\label{global-ED-2}
    E_g (t) & = k_0^+k_1^+k_2^+ \abs{ \widetilde{u} }_{H^1}^2 + k_0^+k_1^-k_2^+ \abs{ \widetilde{v} }_{H^1}^2 + k_0^+k_1^-k_2^- \abs{ \widetilde{p} }_{H^1}^2 + k_0^-k_1^+k_2^+\abs{ \widetilde{q} }_{H^1}^2, \no\\
    D_g (t) & = d_uk_0^+k_1^+k_2^+\abs{ \nabla \widetilde{u} }_{H^1}^2 + d_vk_0^+k_1^-k_2^+ \abs{ \nabla \widetilde{v} }_{H^1}^2 + d_p k_0^+k_1^-k_2^- | \nabla \widetilde{p} |_{H^1}^2 + d_qk_0^-k_1^+k_2^+ | \nabla \widetilde{q} |_{H^1}^2  \no\\
   &\quad + k_0^+ k_2^+ \ov^2 | k_1^+ \widetilde{u} - k_1^-  \widetilde{v} |_{H^1}^2 + k_1^+ k_2^+ | k_0^+ \widetilde{u} - k_0^- \widetilde{q} |_{H^1}^2 + k_0^+ k_1^- \ov^2 | k_2^+ \widetilde{v} - k_2^-  \widetilde{p} |_{H^1}^2.
  \end{align}

Now we state the main energy inequality in global time as follows.
  \begin{lemma}\label{lemma Re-GS-globle-2}
   Assume that $(\widetilde{u}(t,x), \widetilde{v}(t,x), \widetilde{p}(t,x), \widetilde{q}(t,x))$ is the solution to system \eqref{sys:perturb-2} on the interval $[0,T]$ constructed in the Theorem \ref{thm:local-existence-convergence}.  Then there are energy $E_g (u, v, p, q) (t)$ and  energy dissipative rate $D_g (u, v, p,q) (t)$ such that
  \begin{align}\label{energy-global-2}
    \tfrac{1}{2} \tfrac{\d}{\d t} E_g (t) + D_g (t) \leq C_g (1+ E_g^{\frac{1}{2}})E_g^{\frac{1}{2}}D_g
  \end{align}
for all $t \in [0,T]$, where $C_g = 4k_0^+k_1^+k_2^+ + 6 k_0^+k_1^+k_2^+\ov + 4k_0^+k_1^-k_2^+ + 6 k_0^+k_1^-k_2^+\ov$.
\end{lemma}

We substitute $(u, v, p, q)$ for $(\widetilde{u}, \widetilde{v}, \widetilde{p}, \widetilde{q})$ in what follows, for simplicity.

\begin{proof}

We only need to modify the estimates in Lemma \ref{lemma local-Re-GS-1}. More precisely, it is displayed as follows.

\vspace*{2mm}

\smallskip\noindent\textbf{Step 1. $L^2$ estimates.}

We first consider the equality of $u$ in system \eqref{sys:perturb-2}, namely,
  \begin{align}\label{Re-GS-L2g-u-2}
     \tfrac{1}{2} \tfrac{\d}{\d t} \abs{ u }_{L^2}^2 + d_u| \nabla u |_{L^2}^2 & \leq \abs{ v }_{L^6}^2| k_1^+ u - k_1^- v|_{L^2} \abs{ u }_{L^6} + 2 \ov \abs{ u }_{L^6}\abs{ v }_{L^3}| k_1^+ u - k_1^- v |_{L^2} - k_1^+ \ov^2 \skp{u}{u}  \no\\
     & \quad + k_1^- \ov^2 \skp{v}{u} - k_0^+ \skp{u}{u} + k_0^- \skp{q}{u}.
  \end{align}
Then based on the second equation of system \eqref{sys:perturb-2}, we obtain
  \begin{align}\label{Re-GS-L2g-v-2}
     \tfrac{1}{2}\tfrac{\d}{\d t} \abs{ v }_{L^2}^2 + d_v \abs{ \nabla v }_{L^2}^2  & \leq  \abs{ v }_{L^6}^2| k_1^+ u - k_1^- v|_{L^2} \abs{ v }_{L^6} + 2 \ov \abs{ v }_{L^6}\abs{ v }_{L^3}| k_1^+ u - k_1^- v |_{L^2} + k_1^+ \ov^2 \skp{u}{v}  \no\\
    & \quad  - k_1^- \ov^2 \skp{v}{v} - k_2^+ \skp{v}{v} + k_2^- \skp{p}{v}.
  \end{align}
We next consider the equality of $p$, i.e.,
  \begin{align}\label{Re-GS-L2g-p-2}
    \tfrac{1}{2} \tfrac{\d}{\d t} \abs{ p }_{L^2}^2 + d_p \abs{ \nabla p }_{L^2}^2  \leq k_2^{+} \skp{p}{v} - k_2^{-} \skp{p}{p}.
  \end{align}
From the equality of $q$ and the H\"older inequality, we deduce that
  \begin{align}\label{Re-GS-L2g-q-2}
    \tfrac{1}{2} \tfrac{\d}{\d t} \abs{ q}_{L^2}^2 + d_q | \nabla q |_{L^2}^2  \leq k_0^{+} \skp{u}{q} - k_0^{-}\skp{q}{q}.
  \end{align}
Because the fact that
  \begin{align*}
    \skp{k_0^+ u}{k_0^+ u} _{L^2} - 2 \skp{k_1^+u}{k_0^- q}_{L^2} + \skp{k_0^- q}{k_0^- q} _{L^2} &= \skp{k_0^+ u - k_0^- q}{k_0^+ u - k_0^- q} _{L^2}\\
    & = \abs{ k_0^+ u - k_0^- q }_{L^2}^2,\\
    \skp{k_1^+ u}{k_1^+ u} _{L^2} - 2 \skp{k_1^+u}{k_1^- v}_{L^2} + \skp{k_1^- v}{k_1^- v} _{L^2} &= \skp{k_1^+ u - k_1^- v}{k_1^+ u - k_1^- v} _{L^2}\\
    & = \abs{ k_1^+ u - k_1^- v }_{L^2}^2,\\
    \skp{k_2^+ v}{k_2^+ v} _{L^2} - 2 \skp{k_2^+ v }{k_2^- p}_{L^2} + \skp{k_2^- p  }{k_2^- p}_{L^2} &= \skp{k_2^+ v - k_2^- p }{k_2^+ v - k_2^- p} _{L^2}\\
    & = \abs{ k_2^+ v - k_2^- p }_{L^2}^2.
  \end{align*}
  Derived from adding the $k_0^+k_1^+k_2^+$ times of inequalities \eqref{Re-GS-L2g-u-2}, the $k_0^+k_1^-k_2^+$ times of \eqref{Re-GS-L2g-v-2}, the $k_0^+k_1^-k_2^-$ times of \eqref{Re-GS-L2g-p-2} and the $k_0^-k_1^+k_2^+$ times of \eqref{Re-GS-L2g-q-2},  we thereby see that
  \begin{align}\label{Re-GS-L2g-2}
    & \tfrac{1}{2} \tfrac{\d}{\d t} ( k_0^+k_1^+k_2^+ \abs{ u }_{L^2}^2 + k_0^+k_1^-k_2^+ \abs{ v }_{L^2}^2 + k_0^+k_1^-k_2^- \abs{ p }_{L^2}^2 + k_0^-k_1^+k_2^+\abs{ q}_{L^2}^2) + d_uk_0^+k_1^+k_2^+| \nabla u |_{L^2}^2  \no\\
    & \quad + d_vk_0^+k_1^-k_2^+ \abs{ \nabla v }_{L^2}^2 + d_pk_0^+k_1^-k_2^- \abs{ \nabla p }_{L^2}^2  + d_qk_0^-k_1^+k_2^+ | \nabla q |_{L^2}^2  +  k_0^+k_2^+ \ov^2| k_1^+ u - k_1^- v |_{L^2}^2 \no\\
    & \quad + k_1^+k_2^+ | k_0^+ u - k_0^- q |_{L^2}^2 + k_0^+k_1^- | k_2^+ v - k_2^- p |_{L^2}^2  \no\\
    & \leq k_0^+k_1^+k_2^+ \abs{ v }_{H^1}^2 | \nabla u |_{L^2}| k_1^+ u - k_1^- v |_{L^2} + 2 k_0^+k_1^+k_2^+\ov \abs{ v }_{H^1}| \nabla u |_{L^2}| k_1^+ u - k_1^- v |_{L^2}  \no\\
    & \quad + k_0^+k_1^-k_2^+ \abs{ v }_{H^1}^2\abs{ \nabla v }_{L^2}| k_1^+ u - k_1^- v |_{L^2} + 2 k_0^+k_1^-k_2^+\ov \abs{ v }_{H^1}\abs{ \nabla v }_{L^2}| k_1^+ u - k_1^- v |_{L^2}.
  \end{align}

\vspace*{2mm}

\smallskip\noindent\textbf{Step 2. $H^1$ estimates.}

Before our estimate, we know that
  \begin{align*}
    & \p (- k_1^+ uv^2 + k_1^- v^3 -2 k_1^+ \ov uv + 2k_1^- \ov v^2)\\
    & = v^2 (-k_1^+ \p u + k_1^- \p v) + 2 v\cdot \p v (-k_1^+ u + k_1^- v) + 2\ov v (-k_1^+ \p u + k_1^- \p v) + 2\ov \p v (-k_1^+ u + k_1^- v).
  \end{align*}
We first combine the inequality \eqref{Re-GS-L2g-u-2} and take derivative to the first equation of system \eqref{sys:perturb-2}, we obtain
  \begin{align}\label{Re-GS-H1g-u-2}
    \tfrac{1}{2} \tfrac{\d}{\d t} | \p u |_{L^2}^2 + d_u | \p  \nabla u |_{L^2}^2  & \leq \abs{ v }_{L^6}^2 | k_1^+ \p u - k_1^- \p v |_{L^2}| \p u |_{L^6} + 2 \abs{ v }_{L^6} | \p v |_{L^2} | k_1^+ u - k_1^- v |_{L^6}| \p u |_{L^6}  \no\\
    & \quad + 2\ov \abs{ v }_{L^4} | k_1^+ \p u - k_1^- \p v |_{L^2}| \p u |_{L^4} + 2\ov | \p v |_{L^2} | k_1^+ u - k_1^- v |_{L^4}| \p u |_{L^4}  \no\\
    & \quad - k_1^+ \ov^2 \skp{\p u }{\p u} + k_1^- \ov^2 \skp{\p v }{\p u} - k_0^+ \skp{\p u }{\p u} + k_0^- \skp{\p q }{\p u}.
  \end{align}
Next we consider the $H^1$ estimate of $v$, based on the relation \eqref{Re-GS-L2g-v-2}, we see that
  \begin{align}\label{Re-GS-H1g-v-2}
    \tfrac{1}{2} \tfrac{\d}{\d t} | \p v |_{L^2}^2 + d_v | \p  \nabla v |_{L^2}^2 & \leq \abs{ v }_{L^6}^2 | k_1^+ \p u - k_1^- \p v |_{L^2}| \p v |_{L^6} + 2 \abs{ v }_{L^6} | \p v |_{L^2} | k_1^+ u - k_1^- v |_{L^6}| \p v |_{L^6}  \no\\
    & \quad + 2\ov \abs{ v }_{L^4} | k_1^+ \p u - k_1^- \p v |_{L^2}| \p v |_{L^4} + 2\ov | \p v |_{L^2} | k_1^+ u - k_1^- v |_{L^4}| \p v |_{L^4}  \no\\
    & \quad + k_1^+ \ov^2 \skp{\p u }{\p v} - k_1^- \ov^2 \skp{\p v }{\p v} - k_2^+ \skp{\p v }{\p v} + k_2^- \skp{\p p }{\p v}.
  \end{align}
For the equation of $p$, combine with the inequality \eqref{Re-GS-L2g-p-2}, namely,
  \begin{align}\label{Re-GS-H1g-p-2}
    \tfrac{1}{2} \tfrac{\d}{\d t} | \p p |_{L^2}^2 + d_p | \p \nabla p |_{L^2}^2 \leq k_2^+ \skp{\p v }{\p p} - k_2^- \skp{\p p }{\p p} .
  \end{align}
Finally, we consider the inequality \eqref{Re-GS-L2g-q-2},
  \begin{align}\label{Re-GS-H1g-q-2}
    \tfrac{1}{2} \tfrac{\d}{\d t} | \p q |_{L^2}^2 + d_q | \p \nabla q |_{L^2}^2 \leq k_0^+ \skp{\p u }{\p q} - k_0^- \skp{\p q }{\p q}.
  \end{align}

Now we will close the energy estimates,  it is consequently derived from summing the $k_0^+k_1^+k_2^+$ times of inequalities \eqref{Re-GS-H1g-u-2}, the $k_0^+k_1^-k_2^+$ times of inequalities \eqref{Re-GS-H1g-v-2}, the $k_0^+k_1^-k_2^-$ times of inequalities \eqref{Re-GS-H1g-p-2} and the $k_0^-k_1^+k_2^+$ times of inequalities \eqref{Re-GS-H1g-q-2}, we have
  \begin{align*}
    & \tfrac{1}{2} \tfrac{\d}{\d t} ( k_0^+k_1^+k_2^+ | \p u |_{L^2}^2 + k_0^+k_1^-k_2^+ | \p v |_{L^2}^2 + k_0^+k_1^-k_2^- | \p  p |_{L^2}^2 + k_0^-k_1^+k_2^+| \p q |_{L^2}^2) + d_uk_0^+k_1^+k_2^+| \p \nabla u |_{L^2}^2 \no\\
    & \quad + d_vk_0^+k_1^-k_2^+ | \p \nabla v |_{L^2}^2 + d_p k_0^+k_1^-k_2^- | \p \nabla p |_{L^2}^2 + d_qk_0^-k_1^+k_2^+ | \p \nabla q |_{L^2}^2 + k_0^+ k_2^+ \ov^2 | k_1^+ \p u - k_1^- \p v |_{L^2}^2 \no\\
    & \quad + k_1^+ k_2^+ | k_0^+ \p u - k_0^- \p q |_{L^2}^2 + k_0^+ k_1^- | k_2^+ \p v - k_2^- \p p |_{L^2}^2  \no\\
    & \leq 3k_0^+k_1^+k_2^+ \abs{ v }_{H^1}^2 | \nabla u |_{H^1}| k_1^+ u - k_1^- v |_{H^1} + 4 k_0^+k_1^+k_2^+\ov \abs{ v }_{H^1}| \nabla u |_{H^1}| k_1^+ u - k_1^- v |_{H^1}  \no\\
    & \quad + 3k_0^+k_1^-k_2^+ \abs{ v }_{H^1}^2\abs{ \nabla v }_{H^1}| k_1^+ u - k_1^- v |_{H^1} + 4 k_0^+k_1^-k_2^+\ov \abs{ v }_{H^1}\abs{ \nabla v }_{H^1}| k_1^+ u - k_1^- v |_{H^1}.
  \end{align*}

\bigskip\noindent\textbf{Step 3. Close the energy inequality.}

Combine with the $L^2$ estimate and the $H^1$ estimate, we obtain
\begin{align}\label{Re-Gs-H1g-2}
  &\tfrac{1}{2} \tfrac{\d}{\d t} ( k_0^+k_1^+k_2^+ \abs{ u }_{H^1}^2 + k_0^+k_1^-k_2^+ \abs{ v }_{H^1}^2 + k_0^+k_1^-k_2^- \abs{ p }_{H^1}^2 + k_0^-k_1^+k_2^+\abs{ q}_{H^1}^2)  \no\\
   &\quad + d_uk_0^+k_1^+k_2^+\abs{ \nabla u }_{H^1}^2 + d_vk_0^+k_1^-k_2^+ \abs{ \nabla v }_{H^1}^2 + d_p k_0^+k_1^-k_2^- | \nabla p |_{H^1}^2 + d_qk_0^-k_1^+k_2^+ | \nabla q |_{H^1}^2  \no\\
   &\quad + k_0^+ k_2^+ \ov^2 | k_1^+ u - k_1^-  v |_{H^1}^2 + k_1^+ k_2^+ | k_0^+ u - k_0^- q |_{H^1}^2 + k_0^+ k_1^- | k_2^+ v - k_2^-  p |_{H^1}^2  \no\\
   & \leq 4k_0^+k_1^+k_2^+ \abs{ v }_{H^1}^2 | \nabla u |_{H^1}| k_1^+ u - k_1^- v |_{H^1} + 6 k_0^+k_1^+k_2^+\ov \abs{ v }_{H^1}| \nabla u |_{H^1}| k_1^+ u - k_1^- v |_{H^1} \no\\
   &\quad + 4k_0^+k_1^-k_2^+ \abs{ v }_{H^1}^2\abs{ \nabla v }_{H^1}| k_1^+ u - k_1^- v |_{H^1} + 6 k_0^+k_1^-k_2^+\ov \abs{ v }_{H^1}\abs{ \nabla v }_{H^1}| k_1^+ u - k_1^- v |_{H^1}  \no\\
   & \leq C_g (1+ E_g^{\frac{1}{2}})E_g^{\frac{1}{2}}D_g,
\end{align}
where $C_g = 4k_0^+k_1^+k_2^+ + 6 k_0^+k_1^+k_2^+\ov + 4k_0^+k_1^-k_2^+ + 6 k_0^+k_1^-k_2^+\ov $. Consequently, the inequality \eqref{Re-Gs-H1g-2} concludes Lemma \ref{lemma Re-GS-globle-2}.
\end{proof}

\textbf{Proof of Theorem \ref{thm:global-exist-2}.} Based on Lemma \ref{lemma Re-GS-globle-2} and the local-in-time existence to the system \eqref{sys:Re-Gray-Scott}, we can prove the global-in-time existence to the system \eqref{sys:Re-Gray-Scott} with small size initial data near the equilibrium $(\ou, \ov, \op, \oq)$.

Based on Lemma \ref{lemma Re-GS-globle-2}, we now start to prove the global-in-time existence to the system \eqref{sys:Re-Gray-Scott} with small size initial data. We directly deduce
  \begin{align}\label{Global-Final}
    \tfrac{1}{2} \tfrac{\d}{\d t} E_g (u, v, p, q ) (t) +  D_g (u, v, p, q) (t) \leq C_g (1+ E_g^{\frac{1}{2}})E_g^{\frac{1}{2}}D_g
  \end{align}
where
  \begin{align}\label{Cg}
    C_g = 4k_0^+k_1^+k_2^+ + 6 k_0^+k_1^+k_2^+\ov + 4k_0^+k_1^-k_2^+ + 6 k_0^+k_1^-k_2^+\ov.
  \end{align}
We observe that
  \begin{align*}
    E_g(0) = k_0^+k_1^+k_2^+ \abs{ u_0 }_{H^1}^2 + k_0^+k_1^-k_2^+ \abs{ v_0 }_{H^1}^2 + k_0^+k_1^-k_2^- \abs{ p_0 }_{H^1}^2 + k_0^-k_1^+k_2^+\abs{ q_0 }_{H^1}^2.
  \end{align*}
We now take $\nu =\min \{ 1, \tfrac{1}{64 C_g^2} \} \in (0,1]$ such that if $E_g(0) \leq \nu$, then
  \begin{align}\label{small initial bound}
    C_g (1+ E_g^{\frac{1}{2}}(0))E_g^{\frac{1}{2}}(0)\leq \tfrac{1}{4}.
  \end{align}
Now we define
  \begin{align}\label{T}
    T=\sup\{\tau\geq0;\sup\limits_{t\in[0,\tau]}C_g (1+ E_g^{\frac{1}{2}}(t))E_g^{\frac{1}{2}}(t) \leq \tfrac{1}{2}\}\geq 0.
  \end{align}
By the continuity of $E_g(t)$ and the small initial bound \eqref{small initial bound}, that we have $T>0$. We further claim that $T=+\infty$, Indeed, if $T<+\infty$, then the energy inequality in Lemma \ref{lemma Re-GS-globle-2} implies that for all $t\in[0,T]$,
  \begin{align*}
    \tfrac{\d}{\d t}E_g(t)+D_g(t)\leq 0,
  \end{align*}
which means
  \begin{align*}
    \sup\limits_{t\in[0,T]}E_g(t)+\int_0^TD_g(t) \d t\leq E_g(0).
  \end{align*}
Then the above bound reduces to
  \begin{align*}
    \sup\limits_{t\in[0,T]}C_g (1+ E_g^{\frac{1}{2}}(t))E_g^{\frac{1}{2}}(t)\leq C_g (1+ E_g^{\frac{1}{2}}(0))E_g^{\frac{1}{2}}(0)\leq \tfrac{1}{4} < \tfrac{1}{2}.
  \end{align*}
By the continuity of $E_g(t)$, there is a $t^*>0$ such that for all $t\in[0,T+t^*]$,
  \begin{align*}
    C_g (1+ E_g^{\frac{1}{2}}(t))E_g^{\frac{1}{2}}(t)< \tfrac{1}{2},
  \end{align*}
which contradict to the definition of $T$ in \eqref{T}. Thus $T = +\infty$, consequently, we have
  \begin{align*}
    \sup\limits_{t\geq 0}E_g(t)+\int_0^{\infty}D_g(t) \d t\leq E_g(0),
  \end{align*}
which finish proof of the first part of Theorem \ref{thm:global-exist-2}.


\section{Conclusion}

Gray-Scott model is an important reaction-diffusion system, especially in the study of Turing pattern and related issues such as stability/instability, bifurcation and phase transitions. In this paper, we derived by the EnVarA a new reversible Gray-Scott type model. This reversible model possesses a natural entropy structure, and is thus thermodynamically consistent. In physics, our work links non-equilibrium thermodynamics theory on chemical reaction away from equilibrium. From a mathematical viewpoint, this indicates a new possible way to study those chemical reaction-diffusion process from perspectives of modeling, analysis and simulations.


Notice that the spatial domain we work with is the torus or the whole space, in order to avoid more discussions on boundary. Meanwhile a chemical reaction in reality usually occurs in a bounded domain. As mentioned before, it could be the first natural problem for us to derive the corresponding system with proper boundary conditions, by the EnVarA. It is addressed not only for modeling studying, but also for rigorous analysis, say, to consider the global weak solutions for the initial boundary value problems (IBVP).

The second aspect concerns the long time behavior of the obtained solutions, in classical or weak sense and in a bounded or unbounded domain. This is important in studying the stability issues of steady states. Many researches discussed the trend to equilibrium and the convergence rate, see \cite{Ni-11b,BCD-07bull,DFT-17sima,MHM-15jdde} for instance.

The next aspect is with the asymptotic relationship between our reversible Gray-Scott-like model \eqref{sys:Re-Gray-Scott} and the classical Gray-Scott model \eqref{eq:ori-GS}. Recall that on one hand, we have obtained the convergence from reversible system \eqref{sys:Re-Gray-Scott} towards the irreversible system \eqref{sys:Ir-Gray-Scott} (in Theorem \ref{thm:local-existence-convergence}), and on the other hand, we also have provided formally a asymptotic consistency between the irreversible model \eqref{sys:Ir-Gray-Scott} and the classical Gray-Scott model \eqref{eq:ori-GS} (on page 4 in \S \ref{sub:reviews}). Combining the two process together, we actually have addressed a two-step convergence scheme from the reversible system \eqref{sys:Re-Gray-Scott} towards the classical irreversible system \eqref{eq:ori-GS}. Rigorous justification for the limit will involve a slow-fast dynamics perspective and some singular limits. This work is under preparation.

However, we note that as pointed out in \cite{Chu-71ces,GY-11ces}, not all irreversible reactions can be regarded as a limiting case of reversible reactions. This requires more discussions with law of mass action and detailed balance. 

The last but not the least important issue is with the numerical simulation viewpoint, which is very useful in studying the patterns and stability/instability problems, and some coupling effects with other different mechanics such as temperature and electric fields \cite{Mie11,LS-21axv,WLLE-20pre,LWW-21jcp,LWZ-21nnfm}. This could in turn raise more research topics in mathematical analysis.


\section*{Acknowledgments}

This work is partially supported by the National Science Foundation (USA) grants NSF DMS-1950868, the United States-Israel Binational Science Foundation (BSF) \#2024246 (C. Liu, Y. Wang), and the grants from the National Natural Science Foundation of China No. 11971360 and No. 11731008 (N. Jiang), and No. 11871203 (T.-F. Zhang). This work was initiated when T.-F. Zhang visited the Department of Applied Mathematics at Illinois Institute of Technology, he would like to acknowledge the hospitality of IIT and the sponsorship of the China Scholarship Council, under the State Scholarship Fund No. 201906415023.




\end{document}